\documentclass{article}
\usepackage[left=1in,right=1in,top=1in,bottom=1in]{geometry}
\usepackage[utf8]{inputenc}
\usepackage{amsmath}
\usepackage{amsfonts}
\usepackage{amscd}
\usepackage{amssymb,amsthm}
\usepackage{enumerate}
\usepackage{tikz,tikz-cd}
\usepackage{mathtools}
\usepackage{mathrsfs}
\usepackage{authblk}

\theoremstyle{plain}
\newtheorem{thm}{Theorem}[section]
\newtheorem{theorem}[thm]{Theorem}
\newtheorem{introthm}{Theorem}

\newtheorem{introcor}[introthm]{Corollary}
\newtheorem{prop}[thm]{Proposition}
\newtheorem{lem}[thm]{Lemma}
\newtheorem{lemma}[thm]{Lemma}
\newtheorem{cor}[thm]{Corollary}

\theoremstyle{definition}
\newtheorem{defn}[thm]{Definition}

\newtheorem{remark}[thm]{Remark}

\usepackage{tikz}
\usetikzlibrary{shapes,calc,arrows,intersections,decorations.pathreplacing,decorations.markings,shapes.geometric,calc,positioning,backgrounds}

\newcommand{\bC}{{\mathbb C}}

\newcommand{\bN}{{\mathbb N}}

\newcommand{\bR}{{\mathbb R}}

\newcommand{\bZ}{{\mathbb Z}}

\newcommand{\cE}{{\mathcal E}}

\newcommand{\cG}{{\mathcal G}}
\newcommand{\cH}{{\mathcal H}}

\newcommand{\cK}{{\mathcal K}}

\newcommand{\cM}{{\mathcal M}}

\newcommand{\cU}{{\mathcal U}}
\newcommand{\cV}{{\mathcal V}}
\newcommand{\cW}{{\mathcal W}}

\newcommand\Aut{\operatorname{Aut}}

\newcommand\Inn{\operatorname{Inn}}

\DeclareMathOperator{\tr}{tr}

\newcommand{\norm}[1]{\left\|#1\right\|}
\newcommand{\set}[1]{\left\{#1\right\}}
\newcommand{\paren}[1]{\left(#1\right)}

\newcommand{\ang}[1]{\left\langle#1\right\rangle}
\newcommand{\ip}[1]{\left\langle#1\right\rangle}

\newcommand{\glink}[1]{S\paren{#1}}
\newcommand{\gstar}[1]{B\paren{#1}}

\newcommand{\gp}{\mathop{\begin{tikzpicture}[baseline]
\node[circle, fill=cyan, draw=black, inner sep=0pt, minimum size=3pt] (mid) at (0, 2.5pt) {};
\foreach \x in {0,60,...,300} {
\node[circle, fill=cyan, draw, inner sep=0pt, minimum size=3pt] (\x) at ($(mid)!6pt!(mid.\x)$) {};
\draw (\x) -- (mid) ;
}
\end{tikzpicture}}}

\usepackage{xr-hyper}
\usepackage{hyperref}
\makeatletter
\newcommand*{\addFileDependency}[1]{% argument=file name and extension
  \typeout{(#1)}
  \@addtofilelist{#1}
  \IfFileExists{#1}{}{\typeout{No file #1.}}
}
\makeatother

\newcommand*{\myexternaldocument}[1]{%
    \externaldocument{#1}%
    \addFileDependency{#1.tex}%
    \addFileDependency{#1.aux}%
}
\myexternaldocument{main}

\title{\Large\textbf{On the structure of graph product von Neumann algebras}}

\author{Ian Charlesworth\footnote{Yr Ysgol Mathemateg, Prifysgol Caerdydd\hfill \url{charlesworthi@cf.ac.uk}}, 
Rolando de Santiago\footnote{Department of Mathematics and Statistics, California State University, Long Beach \hfill\url{rolando.desantiago@csulb.edu}}, Ben Hayes\footnote{Department of Mathematics, University of Virginia\hfill \url{brh5c@virginia.edu}},\\ 
 David Jekel\footnote{Department of Mathematics, York University\hfill \url{david.jekel@gmail.com}}, Srivatsav Kunnawalkam Elayavalli\footnote{Department of Mathematics, University of California, San Diego\hfill \url{skunnawalkamelayaval@ucsd.edu.}}, 
Brent Nelson\footnote{Department of Mathematics, Michigan State University\hfill \url{brent@math.msu.edu}}}
\date{}

\begin{document}

\maketitle
\begin{abstract}
We undertake a comprehensive study of structural properties of graph products of von Neumann algebras equipped with faithful, normal states, as well as properties of the graph products relative to subalgebras coming from induced subgraphs. Among the  technical contributions in this paper include a complete bimodule calculation for subalgebras arising from subgraphs. 
As an application, we obtain a  complete classification of when two subalgebras coming from induced subgraphs can be amenable relative to each other. We also give complete characterizations of when the graph product can be full, diffuse, or a factor. 
Our results are obtained in a broad generality, and we emphasize that they are new even in the tracial setting. They also allow us to deduce new results about when graph products of groups can be amenable relative to each other.
\end{abstract}

\section*{Introduction}

{ 
Graph products of operator algebras have recently emerged as a subject of intense interest, providing an interpolation between the free product and the tensor product.
The term comes from the group setting where they were introduced by Green \cite{Gr90}; in the operator algebra setting they have been reintroduced and studied under various names by M\l{}otkowski \cite{Mlot2004}, by Speicher and Wysocza\'nski \cite{SpWy2016}, and by Caspers and Fima \cite{CaFi17}.
The mixture of classical and free independence provides a powerful framework for proving results in deformation/rigidity theory \cite{borstcasperssolid, chifan2021cartan, Caspers_2020, graphproductsprimeness, CDDGraphI, CDDGraphII}, the theory of operator space approximation properties of operator algebras \cite{AtkinsonHaagGraph, CaFi17}, and free probability \cite{CC2021, AIMSuperTeamI, Mlot2004, SpWy2016}. 
Graph products of groups are also of significant current interest in group theory \cite{Agol, antolinminasyan, haglund-wise, KoberdaMCG, KimKobBraid, KimKoberdaEmbed, MinasyanOsin}.
Several structural properties of graph product von Neumann algebras -- the Haagerup property, exactness, Connes embeddability, the rapid decay property, absence of Cartan subalgebras, strong solidity, modular theory, and proper proximality -- have also been investigated \cite{AtkinsonHaagGraph, bordenave2023norm,casperscep,Caspers_2020, CaFi17, chifan2021cartan, DKE21}.
Altogether, this makes graph products a natural object to study using tools from geometric group theory, approximation properties, deformation/rigidity theory, free probability, and random matrices.

Given a graph $\cG = (\cV, \cE)$ and family of groups or of von Neumann algebras associated to the vertices, their graph product is a group or von Neumann algebra generated by copies of the input objects with the pairwise relations determined by the graph: two objects connected by an edge should be in direct or tensor product position; two objects not connected by an edge should be in free position.
The relations of higher order must be given as well; we defer the precise definition for von Neumann algebras to Section~\ref{sec:graphproduct}, and refer to Green for the precise definition for groups \cite{Gr90}.
}

In this paper, we undertake a systematic study of precisely when certain natural structural properties of graph products of von Neumann algebras hold.
Moreover, for applications to positions of subalgebras  it is natural to consider algebras corresponding to induced subgraphs and ask when these properties hold ``relative" to another. 
We provide a complete classification for relative amenability, fullness, factoriality and diffuseness (we also completely settle  ``relative diffuseness" i.e. lack of intertwining, in the tracial setting).

%In the free product setting, this assumption was removed through the work of Ueda \cite{UedaTypeIIIfreeproduct}, but since this involves significantly more casework, we leave it as a problem for future research.

Given a graph $\cG=(\cV,\cE)$ and $U\subseteq \cV$, the subgraph induced by $U$ is the graph $\cG|_{U}$ whose vertex set is $U$ and whose edge set is $\cE\cap (U\times U)$.
When we denote $(M,\varphi):=\gp_{v \in \cG} (M_v,\varphi_v)$ and when it is not ambiguous, we will let $(M_U,\varphi_U)$ denote the graph product of $\{(M_{v},\varphi_{v}) :v\in U\}$ with respect to the graph $\cG|_{U}$.

We now state our results characterizing relative amenability.
In the greatest generality our result applies to von Neumann algebras equipped with states which are not necessarily tracial (hereafter referred to as \emph{statial} von Neumann algebras).
However, our results specialize slightly to both the setting of tracial von Neumann algebras and of group von Neumann algebras and therefore to groups; in these more restrictive settings our conditions for relative amenability become slightly nicer to state. Although we state many of our results in the statial setting, they are still new even with an added assumption of traciality.

\begin{introthm}\label{thm: complete char rel amen tracial intro}
Let $\cG=(\cV,\cE)$ be a graph, let $\{(M_v,\tau_v)\colon v\in \cV\}$ be a family of tracial von Neumann algebras, and let $(M,\tau) = \gp_{v \in \cG} (M_v,\tau_v)$.  Assume that $M_{v}$ has a trace zero unitary for every $v\in \cV$.  For $V_1,V_2\subseteq \cV$,  $M_{V_1}$ is amenable relative to $M_{V_2}$ in $M$ if and only if the following occur:
\begin{enumerate}[(1)]
    \item $M_v$ is amenable for each $v \in V_1 \setminus V_2$.
    \item For each $v \in V_1 \setminus V_2$ and $w \in V_1$ with $v \neq w$, either $v$ and $w$ are adjacent or both the following occur:
    \begin{enumerate}[(a)]
        \item $\dim(M_{v})=\dim(M_{w})=2$, \label{item: annoying af counterexample intro}
        \item $v$ and $w$ are adjacent to all vertices in $V_1 \setminus \{v,w\}$.
    \end{enumerate}
\end{enumerate}
\end{introthm}

\noindent Note that Theorem~\ref{thm: complete char rel amen tracial intro} is new even in the case where $M_{v}=L(\Gamma_{v})$ for a family discrete groups $\{\Gamma_{v}:v\in \cV\}$. As such, we obtain a complete classification of when subgroups corresponding to induced subgraphs can be amenable relative to each other for graph products of groups, which follows immediately from Theorem~\ref{thm: complete char rel amen tracial intro}.

\begin{introcor}
 Let $\mathcal{G}=(\mathcal{V},\mathcal{E})$ be a graph and $\{\Gamma_{v}:v\in \mathcal{V}\}$ be a family groups. Let $\Gamma$ be the graph product of  $\{\Gamma_{v}:v\in \mathcal{V}\}$ with respect to $\mathcal{G}$, and for $U\subseteq \mathcal{V}$ let $\Gamma_{U}$ be the graph product of $\{\Gamma_{v}:v\in U\}$ with respect to $\mathcal{G}|_{U}$. Then for $V_{1},V_{2}\subseteq V$ we have that $\Gamma_{V_{1}}$ is amenable relative to $\Gamma_{V_{2}}$ inside $\Gamma$ if and only if both of the following occur:
\begin{enumerate}[(1)]
\item $\Gamma_{v}$ is amenable for each $v\in V_{1}\setminus V_{2}$ 
\item for each $v\in V_{1}\setminus V_{2}$ and $w\in V_{1}$ with $v\neq w$, either $(v,w)\in \cE$ or both of the following occur:
 \begin{enumerate}[(a)]
        \item $\Gamma_{v}\cong \Gamma_{w}\cong \bZ/2\bZ$
        \item $v$ and $w$ are adjacent to all vertices in $V_1 \setminus \{v,w\}$.
    \end{enumerate}
\end{enumerate}
\end{introcor}

\noindent We also have a complete result in the statial setting that generalizes Theorem \ref{thm: complete char rel amen tracial intro}. This also yields a complete characterization of when the graph product von Neumann algebra is amenable (see Proposition~\ref{prop: amenability characterization}).

\begin{introthm}\label{thm: complete char rel amen stacial intro}
Let $\cG=(\cV,\cE)$ be a graph, let $\{(M_v,\varphi_v)\colon v\in \cV\}$ be a family of statial von Neumann algebras, and let $(M,\varphi) = \gp_{v \in \cG} (M_v,\varphi_v)$.  Assume $M_v^{\varphi_{v}}$ has a state zero unitary for every $v\in \cV$.  For $V_1,V_2\subseteq\cV$, $M_{V_1}$ is amenable relative to $M_{V_2}$ inside $M$ if and only if both of the following occur:
\begin{enumerate}[(1)]
    \item $M_v$ is amenable for each $v \in V_1 \setminus V_2$. \label{item: amenable vertex wise intro}
    \item \label{item: rel amen free product intro} For each $v \in V_1 \setminus V_2$ and $w \in V_1$ with $v \neq w$, either $v$ and $w$ are adjacent or both of the following occur:
    \begin{enumerate}[(a)]
        \item $M_{\{v,w\}} = M_v * M_w$ is either amenable if $w \not \in V_2$ or is amenable relative to $M_w$ if $w \in V_2$.\label{item: case of two vertices relative amen intro}
        \item $v$ and $w$ are adjacent to all vertices in $V_1 \setminus \{v,w\}$.\label{item: relative amen adjacency constraint intro}
    \end{enumerate}
\end{enumerate}
\end{introthm}

\noindent Our assumption that the centralizer subalgebra $M_v^{\varphi_v}$ admits a  {state zero unitary} is mild (see Appendix~\ref{sec: state zero} for a characterization in terms of minimal  central  projections). 
Indeed, in the tracial setting this holds for any nontrivial group von Neumann algebra, any diffuse algebra, and any finite factor but $\bC$.
Moreover, this assumption in the non-tracial setting already appears in foundational works in the non-tracial setting \cite{Bar95, DimaFAW}.
While this was removed in the free product setting through the work of Ueda \cite{UedaTypeIIIfreeproduct}, doing so in our setting is likely to involve significant effort which we leave for future investigation.

A significant tool for our study of relative amenability is to apply work of \cite{BMOCoAmenable} (building off of prior work of \cite{ADAmenableCorr,HaagerupForms}), which states that there is a (not assumed to be normal) conditional expectation $\ip{M,e_{Q}}\to N$ when $N,Q\leq M$ are with expectation if and only if
\[_{N}L^{2}(M)_{N}\prec _{B}L^{2}(\ip{M,e_{Q}})_{N}.\]
Strictly speaking, the existence of such a conditional expectation is different from $N$ being amenable relative to $Q$ inside of $M$, but this turns out to be not a problem.
Since we may view $L^{2}(\ip{M,e_{Q}})$ as a relative tensor product, it thus makes sense to address relative amenability via understanding the bimodule structure of $_{M_{V_{1}}}L^{2}(M)_{M_{V_{2}}}$ for $V_{1},V_{2}$ subsets of the vertices, as well as the fusion rules for such bimodules.
We obtain a complete description of such bimodules and their fusion rules in terms of the combinatorial structure of the graph, and the dimensions of the algebras attached to the vertices.

\begin{introthm} \label{thm: bimodules intro}
Let $\cG=(\cV,\cE)$ be a graph, let $\{(M_v,\varphi_v)\colon v\in \cV\}$ be a family of statial von Neumann algebras, and let $(M,\varphi) = \gp_{v \in \cG} (M_v,\varphi_v)$. For $V_1, V_2 \subseteq \cV$ one has
    \begin{equation} \label{eq: bimodule multiplicity intro}
        _{M_{V_1}} L^2(M,\varphi)_{M_{V_2}} \cong \bigoplus_{U \subseteq V_1 \cap V_2} (_{M_{V_1}} L^2(M_{V_1},\varphi_{V_1}) \underset{M_U}{\otimes} L^2(M_{V_2},\varphi_{V_2})_{M_{V_2}})^{\oplus k_{\cG}(V_1,V_2,U)}
    \end{equation}
where $k_{\cG}(V_1,V_2,U)$ is explicitly determined in terms of the graph structure and dimension of the vertex algebras (see Theorem \ref{thm: bimodules} for the precise description). Moreover, we have the following fusion rules.
For $V_1, V_2 \subseteq \cV$ and $U \subseteq V_1 \cap V_2$, let
\[
    \mathscr{H}_U(V_1,V_2) = _{M_{V_1}} L^2(M_{V_1},\varphi_{V_1}) \underset{M_{U}}{\otimes} L^2(M_{V_2},\varphi_{V_2})_{M_{V_2}}
\]
Then for $U_1 \subseteq V_1 \cap V_2$ and $U_2 \subseteq V_2 \cap V_3$,
\[
\mathscr{H}_{U_1}(V_1,V_2) \otimes_{M_{V_2}} \mathscr{H}_{U_2}(V_2,V_3)
\cong
\bigoplus_{W \subseteq U_1 \cap U_2} \mathscr{H}_W(V_1,V_3)^{\oplus k_{\cG_2}(U_1,U_2,W)},
\]
where $\cG_2$ is the subgraph of $\cG$ induced by $V_2$.
\end{introthm}

\noindent Theorem~\ref{thm: bimodules intro} is proved in two parts in the body of the paper, in Theorem~\ref{thm: bimodules} and Proposition~\ref{prop: fusion}.
The utility of such a precise computation can be seen from Theorem~\ref{thm: weakly coarse}, which provides a very easy to check characterization of when certain bimodules are weakly coarse.

We also give a complete characterization of fullness, factoriality, and diffuseness.

\begin{introthm}\label{thm: char diffuse factor full intro}
Let $\cG=(\cV,\cE)$ be a graph, let $\{(M_v,\varphi_v)\colon v\in \cV\}$ be a family of statial von Neumann algebras, and let $(M,\varphi) = \gp_{v \in \cG} (M_v,\varphi_v)$.  Assume $M_v^{\varphi_{v}}$ has a state zero unitary for every $v\in \cV$.
\begin{enumerate}[(1)]
    \item $M$ is diffuse if and only if either: (a) some $M_v$ is diffuse; or (b) $\cG$ is not a complete graph.
    \label{item: diffuse characterization intro}
    
    \item $M$ is a factor if and only if both: (a) whenever a vertex $v$ is adjacent to all other vertices of $\cG$, then $M_v$ is a factor; and (b) if $v$ and $w$ are not adjacent to each other but are adjacent to all other vertices of $\cG$, then $\max(\dim{M_{v}}, \dim{M_{w}})\geq 3$.  
    \label{item: factor charcaterization intro}
    
    \item   $M$ is full if and only if both: (a) whenever a vertex $v$ is adjacent to all other vertices of $\cG$, then $M_v$ is full; and (b) if $v$ and $w$ are not adjacent to each other but are adjacent to all other vertices of $\cG$, then $\max(\dim{M_v}, \dim{M_w}) \geq 3$. \label{item: full characterization intro}
\end{enumerate}
\end{introthm}

For tracial algebras, we also provide a complete characterization of relative diffuseness (or lack of intertwining) of $M_{V_{1}}$ relative to $M_{V_{2}}$ analogous to Theorem~\ref{thm: char diffuse factor full intro}.(\ref{item: diffuse characterization intro}). We refer the reader to Proposition \ref{prop: relative diffuseness} for the relevant statement, which amounts to in (a) requiring that a diffuse algebra be attached to a vertex in $V_{1}\setminus V_{2}$, and replacing the ``non-completeness" in (b) with the lack of edge between a vertex in $V_{1}\setminus V_{2}$ with a vertex in $V_{1}\cap V_{2}$.

\subsection*{Acknowledgements}
We thank the American Institute of Mathematics SQuaRES program for hosting us in Aprils 2023 and 2024 to collaborate on this project. BH was supported by NSF CAREER grant DMS-2144739.
DJ was supported by an NSF postdoctoral fellowship DMS-2002826, the Fields Institute, and acknowledges the support of the Natural Sciences and Engineering Research Council of Canada (NSERC).
SKE was supported by NSF grant DMS-2350049. BN was supported by NSF grant DMS-2247047.

\tableofcontents

\section{Preliminaries}\label{sec: preliminaries}

\subsection{$\cG$-independence and graph products}\label{sec:graphproduct}
Throughout, a \emph{graph} is a pair $(\cV, \cE)$ where $\cV$ is a finite set of \emph{vertices} and $\cE \subseteq \cV \times \cV$ is a set of \emph{edges} such that $(u, v) \in \cE$ if and only if $(v, u) \in \cE$; we also insist that $(u, u) \notin \cE$ for all $u \in \cV$.
In other words, our graphs are finite and simple (undirected, and without self-loops).
We write $v\sim w$ (respectively, $v\not\sim w$) whenever $(v,w)\in \cE$ (respectively, $(v,w)\not\in \cE$); we make the dependence on the graph implicit.
For a given $v\in \cV$, we denote the sphere centered at $v$ by $\glink{v}=\{w\in \cV: w\sim v \},$ and the ball centered at $v$ by $\gstar{v}:=\glink{v}\cup \{v\}$.

A word $v_1 \cdots v_n$ in the alphabet $\cV$ is said to be \emph{$\cG$-reduced} if whenever $i < k$ with $v_i = v_k$, there is some $i < j < k$ so that $(v_i, v_j) \notin \cE$.
(By repeatedly applying this condition, we could further assume that $v_i \neq v_j$.)

Suppose that $\cG = (\cV, \cE)$ is a graph and $(M, \varphi)$ is a statial von Neumann algebra.
    For each $v \in \cV$, let $1 \in M_v \subseteq M$ be a unital $*$-subalgebra.
    Then the family $\{M_v:v \in \cV\}$ is said to be \emph{$\cG$-independent} if whenever $v_1 \cdots v_n$ is a $\cG$-reduced word and $x_1, \ldots, x_n \in M$ with $x_i \in \ker(\varphi) \cap M_{v_i}$, we have 
        \[
            \varphi(x_1\cdots x_n) = 0.
        \]
On the other hand, given a graph $\cG = (\cV, \cE)$ and a family of statial von Neumann algebras $\{(M_v, \varphi_v):v \in \cV\}$, there is up to isomorphism a unique statial von Neumann algebra $(M, \varphi)$ and state-preserving inclusions $M_v \hookrightarrow M$ so that the images of the $M_v$ are $\cG$-independent and generate $M$.
We refer to this algebra $(M, \varphi)$ as the \emph{graph product} of the family $\{(M_v, \varphi_v): v\in \cV\}$ and write \[(M, \varphi) = \gp_{v \in \cG} (M_v, \varphi_v).\]

The existence and uniqueness of the graph product was shown by M\l{}otkowski and also by Caspers and Fima; moreover, if each $\varphi_v$ is tracial then so is the state on the graph product \cite{Mlot2004, CaFi17}.

\subsection{Structural properties of von Neumann algebras}

We recall the definitions of the structural properties appearing in the theorems in the introduction of the paper.

A von Neumann algebra $M$ is said to be \emph{full} if whenever a bounded net $(x_i)_{i\in I}\subset M$ satisfies $\| \varphi([x_i,\, \cdot\ ])\|\to 0$ for all $\varphi\in M_*$ then there exists a net of scalars $(\lambda_i)_{i\in I}\subset \bC$ such that $(x_i - \lambda_i) \to 0$ strongly. This notion was introduced for von Neumann algebras with separable preduals by Connes in \cite{ConnesFull}, where he showed it was equivalent to $\Inn(M)$ being closed in $\Aut(M)$ under the point norm topology \cite[Theorem 3.5]{ConnesFull}. \cite{HMVUltra,AndoHaagerup} considered this notion in the more general $\sigma$-finite case, where they showed it was equivalent to $M'\cap M^\omega=\bC$ \cite[Proposition 4.35, Theorem 5.2]{AndoHaagerup}, \cite[Corollary 3.7]{HMVUltra}. Here $M^\omega$ denotes the Ocneanu ultrapower (see Appendix~\ref{sec: Ocneaun ultrapowers}), and in this paper we will always verify fullness by proving $M'\cap M^\omega=\bC$. We note that the proof of this implication can be found in \cite[Proposition 4.35]{AndoHaagerup}, and in fact it is an exercise from \cite[Proposition 2.8]{ConnesFull}.

Let  $A,B\leq M$ be inclusions of von Neumann algebras with conditional expectations $E_A, E_B$. Let $\ip{M,e_B}$ denote the basic construction associated to the inclusion $(B\subset M, E_B)$. We say that $A$ is \emph{amenable relative to $B$ inside $M$} if there exists a conditional expectation $\Phi\colon \ip{M,e_B} \to A$ such that $\Phi|_M$ is normal \cite{PopaCorr} (see also \cite{PopaQR} and \cite[Definition 4]{PoMonodRelativeAmen}).

\section{ Diffuseness, factoriality, and fullness}\label{sec: diff fact full}

In this section, we classify when  a graph product $\mathrm{W}^*$-algebra has various properties (diffuseness, amenability, factoriality, fullness) based on the input algebras $M_v$ (see \cite[Corollary 2.29]{CaFi17} for a partial result in this direction).

We will use the graph join operation to produce a tensor product decomposition for the graph product, thereby reducing the study of various properties of the graph product over $\mathcal{G}$ to the properties of the subgraphs $\mathcal{G}_1$, \dots, $\mathcal{G}_n$. Given graphs $\mathcal{G}_j = (\mathcal{V}_j,\mathcal{E}_j)$ for $j = 1$, \dots, $n$, the \emph{graph join} $\mathcal{G}_1 + \mathcal{G}_2 + \dots + \mathcal{G}_n$ is the graph obtained from the disjoint union of $\mathcal{G}_1$, \dots, $\mathcal{G}_n$ by adding edges from every vertex of $\mathcal{G}_i$ to every vertex of $\mathcal{G}_j$ for $i \neq j$.  We say that $\mathcal{G}$ is \textbf{join-irreducible} if it is nonempty and cannot be decomposed as a graph join of two nonempty graphs.  By \cite[Theorem 1]{JoinDecomposition}, every graph $\mathcal{G}$ has a unique (up to permutation) decomposition as $\mathcal{G}_1 + \dots + \mathcal{G}_n$, where $\mathcal{G}_1$, \dots, $\mathcal{G}_n$ are join-irreducible (here we allow a single vertex to be considered as a join-irreducible graph). The next proposition follows immediately from the definition of the graph product for statial von Neumann algebras.

\begin{prop}\label{prop: graph join decomp to tensor decomp}
Let $\cG=(\cV,\cE)$ be a graph and let $\{(M_v,\varphi_v):v \in \cV\}$ be a family of stacial von Neumann algebras. If $\mathcal{G} = \mathcal{G}_1 + \dots + \mathcal{G}_n$ for graphs $\cG_j=(\cV_j,\cE_j)$, $j=1,\ldots, n$, then
\[
\gp_{v \in \mathcal{G}} (M_v,\varphi_v) = \underset{1\leq j\leq n}{\overline{\bigotimes}} \ \gp_{v \in \mathcal{G}_j} (M_v,\varphi_v) .
\]
\end{prop}

Since it is known that diffuseness, factoriality, and fullness of a tensor product can be characterized in terms of the corresponding properties for the tensor factors (see the proof of Theorem~\ref{thm: char diffuse factor full intro} in Section~\ref{sec: proof of char diffuse factor full intro} below), the above proposition allows us to reduce our analysis to the join-irreducible case.
The general outline of the argument is as follows.  By the foregoing argument, we reduce to the case when $\mathcal{G}$ is join-irreducible, then further divide into cases based on whether the number of vertices of $\mathcal{G}$ is $1$, $2$, or greater than $2$, and decide diffuseness, amenability, factoriality, or fullness in each case.
Of course, if $\mathcal{G}$ consists of a single vertex $v$, then this is simply the diffuseness, amenability, factoriality, or fullness of the input algebra $M_v$, and so we will only address the cases of $|\cV|=2$ and $|\cV|\geq 3$ below.
If $\mathcal{G}$ has two vertices, then these two vertices must not be connected by an edge, because otherwise $\mathcal{G}$ would decompose as the graph join of the two vertices.  Hence, $\gp_{v \in \mathcal{G}} (M_v,\varphi_v)$ is the free product $(M_1,\varphi_1) * (M_2,\varphi_2)$ of the two input algebras.  Now, if we assume that $M_1^{\varphi_1}$ and $M_2^{\varphi_2}$ each contain state zero unitaries $u_1$ and $u_2$, then by free independence the product $u_1 u_2$ will be a Haar unitary in $(M_1,\varphi_1) * (M_2,\varphi_2)$, and hence $(M_1,\varphi_1) * (M_2,\varphi_2)$ is diffuse.
If $M_1 \cong M_2 \cong \bC \oplus \bC$ with equal weight on each of the two summands, then $M_1 * M_2$ is amenable and not a factor, and in all other cases (under the assumption that $M_1^{\varphi_1}$ and $M_2^{\varphi_2}$ admit state zero unitaries), it is a full factor by results of Ueda \cite{UedaTypeIIIfreeproduct}.
The remaining case is then when $\mathcal{G}$ has at least three vertices, which we will handle separately as a general argument.  

Before proceeding in this way, we first observe a combinatorial condition that follows from a lack of graph join decomposition.

\begin{lemma} \label{lem:graphdecomposition}
Let $\mathcal{G}=(\mathcal{V},\mathcal{E})$ be a join-irreducible graph. Then either $\mathcal{G}$ is disconnected or for every vertex $v_0 \in \mathcal{V}$, there exist $v_1, v_2 \in \mathcal{V} \setminus \{v_0\}$ such that
    \[
        v_0 \sim v_1,\quad v_0 \not \sim v_2,\quad v_1 \not \sim v_2.
    \]
\end{lemma}

\begin{proof}
We proceed by contrapositive.  Suppose that $\mathcal{G}$ is connected and that there exists a vertex $v_0$ such that for all $v_1, v_2 \in \mathcal{V} \setminus \{v_0\}$, if $v_1 \sim v_0$ and $v_2 \not \sim v_0$, then $v_1 \sim v_2$.  Fix such a $v_0$.  Let $S = \gstar{v_0}$.  We claim that every vertex in $S$ is adjacent to every vertex in $S^c$.  Let $v \in S$ and $w \in S^c$.  If $w = v_0$, then $w \sim v$ by definition of $\gstar{v_0}$.  If $w \neq v_0$, then because $w \not \sim v_0$ and $v \sim v_0$, we have $v \sim w$.  Since every vertex in $S$ is connected to every vertex in $S^c$, we can decompose $\mathcal{G}$ as the graph join of the two induced subgraphs with vertex sets $S$ and $S^c$.
\end{proof}

\begin{remark}
The converse of this lemma does not hold.  In fact, suppose that we take graphs $\mathcal{G}_1$ and $\mathcal{G}_2$ which both satisfy that for every $ v_0 \in \mathcal{V}$, there exists $v_1, v_2 \in \mathcal{V} \setminus \{v_0\}$ such that $v_0 \sim v_1, v_0 \not \sim v_2, v_1 \not \sim v_2$.  Then $\mathcal{G}_1 + \mathcal{G}_2$ also satisfies this condition.  More generally, if $\mathcal{V}$ is expressed as a union of subsets $V_j$, and the subgraphs induced by $V_j$ have this property, then the whole graph has this property.
\end{remark}

The following is a special case of Theorem~\ref{thm: char diffuse factor full intro} for join-irreducible graphs, which will be used in the general proof in conjunction with strategy outlined after Proposition~\ref{prop: graph join decomp to tensor decomp}. The proof makes use of Ocneanu ultrapowers and some related lemmas which are detailed in Appendix~\ref{sec: Ocneaun ultrapowers}. It also uses the fact that subalgebras $M_U$ corresponding to induced subgraphs admit unique state preserving, faithful, normal, conditional expectations $E_{M_U}\colon M\to M_U$ (see \cite[Remark 2.14]{CaFi17}). The uniqueness implies, in particular, that $M_{V_1\cap V_2}, M_{V_1}, M_{V_2}, M$ forms a commuting square for any subsets $V_1,V_2\subset \cV$.

\begin{thm} \label{thm:factoriality}
Let $\mathcal{G}=(\mathcal{V},\mathcal{E})$ be a join-irreducible graph. Let $\{(M_v,\varphi_v):v \in \cV\}$ be a family of stacial von Neumann algebras and let $(M,\varphi)=\gp_{v\in \mathcal{G}}(M_{v},\varphi_{v})$. 
Assume $M_v^{\varphi_{v}}$ has a state zero unitary for every $v\in \cV$.
\begin{itemize}
 \item If $|\mathcal{V}|=2$ with $\mathcal{V}=\{v,w\}$ and $\dim(M_{v})=\dim(M_{w})=2$, then $M$ is diffuse but not a factor.
\item If $|\mathcal{V}|=2$,  and $\max(\dim(M_{v}),\dim(M_{w}))\geq 3$, then $M$ is a diffuse full factor.
 \item If $|\mathcal{V}|\geq 3$, then $M$ is a diffuse full factor. 
\end{itemize}
\end{thm}
\begin{proof}

First suppose $\mathcal{V}=\{v_{1},v_{2}\}$ and recall that join-irreducibility of $\cG$ implies $(M,\varphi)=(M_{v_{1}},\varphi_{v_1})*(M_{v_{2}},\varphi_{v_2})$. If one of $M_{v_{1}}$ or $M_{v_{2}}$ has dimension at least $3$,  then $M$ is diffuse and a full factor by \cite[Theorem 4.1 and Remark 4.2]{UedaTypeIIIfreeproduct}. If $\dim(M_{v_{1}})=\dim(M_{v_{2}})=2$ so that $M_{v_i}\cong \bC\oplus \bC$ for $i=1,2$, then $\varphi_{v_1}, \varphi_{v_2}$ are necessarily tracial and our assumption on the existence of trace zero unitaries forces these traces to put equal weight on each factor of $\bC$. Hence $M$ is diffuse but is not a factor by \cite[Theorem 1.1]{DykemaFreeproductsHyper}. 

We now assume that $|\mathcal{V}|\geq 3$. Note that if a von Neumann algebra $P$ has a normal conditional expectation onto a diffuse subalgebra, then $P$ is diffuse (this follows from restricting such a conditional expectation to the maximal purely atomic direct summand of $P$ and applying \cite[Theorem IV.2.2.3]{BlackadarOA}). Since we already have normal conditional expectations onto subalgebras corresponding to induced subgraphs, it follows from the above paragraph that $M$ is diffuse in this case. So we only focus on proving $M$ is a full factor.
By Lemma \ref{lem:graphdecomposition}, it suffices to prove the theorem under the weaker condition that either $\mathcal{G}$ is disconnected or for every $v_0 \in \mathcal{V}$, there exist $v_1$, $v_2 \in \mathcal{V} \setminus \{v_0\}$ such that $v_0 \sim v_1$, $v_0 \not \sim v_2$, $v_1 \not \sim v_2$. 

Suppose $\mathcal{G}$ is disconnected and $|\mathcal{V}| \geq 3$. Then there exists a vertex $v_0$ and two other vertices $v_1$ and $v_2$ that are not in the same connected component as $v_0$.  Let $V_0\subset \cV$ be the vertices in the connected component of $\mathcal{G}$ containing $v_0$.  Then
\[
    M= M_{V_0} * M_{\cV \setminus V_0}.
\]
Let $u_0$, $u_1$, and $u_2$ be trace zero unitaries in $M_{v_0}$, $M_{v_1}$, and $M_{v_2}$ respectively.  We have $\varphi(u_1^*u_2)= \varphi(u_1^*)\varphi(u_2)=0$ in both cases $v_1\sim v_2$ and $v_1\not\sim v_2$.
Thus, the unitaries satisfy the hypotheses of Lemma \ref{lem:freeproductcenter} with $B = \bC$.  It follows that for every cofinal ultrafilter $\omega$ on a directed set, we have $M' \cap M^\omega \subseteq \bC^\omega = \bC$, so that $M$ is full.

Now consider the case where for every $v_0 \in \mathcal{V}$, there exist $v_1$, $v_2 \in \mathcal{V} \setminus \{v_0\}$ such that $v_0 \sim v_1$, $v_0 \not \sim v_2$, $v_1 \not \sim v_2$. (In this case automatically $|\mathcal{V}| \geq 3$.) Fix a cofinal ultrafilter $\omega$ on a direct set, and  a vertex $v_0$. Note that by \cite[Theorem 2.26]{CaFi17}
\[
M = M_{\gstar{v_0}} *_{M_{\glink{v_0}}} M_{\cV \setminus \{v_0\}}.
\]
Let $v_1$ and $v_2$ be vertices with $v_1 \sim v_0$, $v_2 \not \sim v_0$, $v_1 \not \sim v_2$.  Let $u_0$, $u_1$, and $u_2$ be a trace zero unitaries from $M_{v_0}^{\varphi_{v_{0}}}$, $M_{v_1}^{\varphi_{v_{1}}}$, and $M_{v_2}^{\varphi_{v_{2}}}$ respectively.  We want to apply Lemma \ref{lem:freeproductcenter} to the unitaries $u_0$, $u_2$, and $u_1^* u_2 u_1$.  Note that the words $v_0$, $v_2$, $v_1v_2v_1$, and $v_1v_2v_1v_2$ are reduced and each have some element not in $\glink{v_0}$; therefore, 
by the alternating expectation condition defining free independence with amalgamation
\[
    E_{M_{\glink{v_0}}}[u_0] = E_{M_{\glink{v_0}}}[u_2] = E_{M_{\glink{v_0}}}[u_1u_2u_1^*] = E_{M_{\glink{v_0}}}[(u_1^*u_2u_1)^*u_2] = 0.
\]
Moreover, $u_{1}u_{2}^{*}u_{1}$ is in the centralizer of $M_{\cV\setminus\{v_{0}\}}$. 
Therefore, by Lemma \ref{lem:freeproductcenter},
\[
M' \cap M^\omega \subseteq (M_{\glink{v_0}})^\omega.
\]
Now the vertex $v_0$ was arbitrary, and therefore, by Lemma \ref{lem: commuting squares and intersections in ups}
\[
M' \cap M^\omega \subseteq \bigcap_{v_0 \in V} M_{\glink{v_0}}^\omega = \left( \bigcap_{v_0 \in V} M_{\glink{v_0}} \right)^\omega.
\]
By \cite[Proposition 2.25]{CaFi17},
    \[
        \bigcap_{v_0 \in V} M_{\glink{v_0}} = M_{\bigcap_{v_0 \in V} \glink{v_0}}.
    \]
Because $v_0 \not \in \glink{v_0}$ by definition, we have $\bigcap_{v_0 \in V} \glink{v_0} = \varnothing$.  Hence, $M' \cap M^\omega \subseteq \bC$, so that $M$ is full. \qedhere 
\end{proof}

\begin{remark}
In particular, suppose that the graph $\mathcal{G}$ has diameter at least $3$, meaning that there exists two vertices $v$ and $w$ with distance at least $3$ in the graph.  Then $\mathcal{G}$ is join-irreducible because in a graph join any two vertices have distance at most $2$.  Therefore, the theorem implies that $\gp_{v \in \mathcal{G}} (M_v,\varphi_v)$ is a full factor provided that each $M_v^{\varphi_v}$ contains a state zero unitary.
\end{remark}

Consider a non-join-irreducible graph $\cG$ and suppose $\cG=\cG_1+\cdots +\cG_n$ is its graph join decomposition for graphs $\cG_j=(\cV_j,\cE_j)$. Since diffuseness, factoriality, and fullness are all automatic for graph products over $\cG_j$ when $|\cV_j|\geq 3$, to understand these properties for graph products over $\cG$ it is not necessary to compute its entire graph-join decomposition.  We merely need to be able to locate the $\mathcal{G}_j$'s that have $1$ or $2$ vertices.  For this purpose, we record the following observation:

\begin{lem} \label{lem: join component characterization}
Let $v$ be a vertex of a graph $\mathcal{G}=(\mathcal{V},\mathcal{E})$.  Then $v$ comprises one of the components in the graph join decomposition of $\mathcal{G}$ if and only if $v$ is adjacent to all the other vertices of $\mathcal{G}$.

Similarly, let $v$ and $w$ be distinct vertices of $\mathcal{G}$.  Then $\{v,w\}$ comprises one of the components in the graph join decomposition of $\mathcal{G}$ if and only if $v$ and $w$ are not adjacent to each other but are adjacent to all the other vertices in $\mathcal{G}$.
\end{lem}

We remark that detecting components in the graph join decomposition of $\mathcal{G}$ with one or two vertices is algorithmically much simpler than finding the full graph join decomposition (it can be done in polynomial time in the number of vertices).

\subsection{Proof of Theorem~\ref{thm: char diffuse factor full intro}}\label{sec: proof of char diffuse factor full intro}

Let $\cG=\mathcal{G}_1 + \dots + \mathcal{G}_n$ be the graph join decomposition for graphs $\cG_j=(\cV_j, \cE_j)$, $j=1,\ldots,n$. Denote $(N_j,\psi_j):=\gp_{v\in \cG_j} (M_v,\varphi_v)$ for each $j=1,\ldots, n$, so that
    \[
        (M,\varphi) \cong (N_1,\psi_1)\bar\otimes \cdots \bar\otimes (N_n,\psi_n)
    \]
by Proposition~\ref{prop: graph join decomp to tensor decomp}.\\

\noindent(\ref{item: diffuse characterization intro}): $M$ is diffuse if and only if $N_j$ is diffuse for some $j$. If $\mathcal{G}_j$ has at least two vertices, then 
$N_j$ is diffuse by Theorem \ref{thm:factoriality}. Thus, the only way $M$ can fail to be diffuse is if all the $\mathcal{G}_j$'s are singletons (that is, $\mathcal{G}$ is a complete graph), and none of the $M_v$'s are diffuse.\\

\noindent(\ref{item: factor charcaterization intro}) $M$ is a factor if and only if $N_j$ is a factor for each $j=1,\ldots, n$. If $\cG_j$ has at least three vertices, then $N_j$ is automatically a factor by Theorem~\ref{thm:factoriality}. So for $M$ to be a factor it is necessary and sufficient that $N_j$ is a factor whenever $|\cV_j|\leq 2$. For $\cV_j=\{v\}$, this reduces to $M_v$ being a factor, and the characterization of singleton components in Lemma~\ref{lem: join component characterization} this yields condition (a). For $\cV_j=\{v,w\}$, $N_j$ is a factor if and only if $\max(\dim(M_{v}),\dim(M_{w}))\geq 3$ by Theorem~\ref{thm:factoriality}, and the characterization of two-element components in Lemma~\ref{lem: join component characterization} this yields condition (b).\\

\noindent(\ref{item: full characterization intro}) $M$ is full if and only if $N_j$ is full for each $j=1,\ldots, n$ by  \cite[Corollary 2.3]{Connes}, \cite[Corollary B]{HMVUltra}. Noting that the characterization of fullness coincides with that of factoriality for join-irreducible graphs in Theorem~\ref{thm:factoriality}, the same argument used in the previous part completes the proof.

\begin{remark}
Observe that under our standard assumption that $M_v^{\varphi_v}$ admits a state zero unitary, the graph product over $\cG$ gives a non-full factor if and only if there exists $v\in \cV$ adjacent to every other vertex with $M_v$ a non-full factor. Indeed, using the notation of the above proof, $M$ is a non-full factor if and only if each $N_j$ is a factor and at least one, say $N_{j_0}$, is non-full. According to Theorem~\ref{thm:factoriality}, this is only possible if $\cV_{j_0}$ consists of a single vertex and the algebra over that vertex is a non-full factor.
\end{remark}

\section{Relatively reduced words and conditional expectations}\label{sec: reduced words}

M\l otkkowsi \cite{Mlot2004} and Caspers--Fima \cite{CaFi17} used reduced words to give a description of the standard form of a graph product is analogous to a Fock space. From their description, one can build an orthonormal basis for $L^{2}$ of the graph product using an orthonormal basis of the vertex algebras. In Section~\ref{sec: fusion hah!}, we will have to describe the standard form of the graph product as a bimodule over two subalgebras coming from subgraphs. In order to investigate relative amenability in Section~\ref{sec: rel amen}, we will also have to describe the fusion rules. In this bimodule situation it is natural to look for (an analogue of) a \emph{Pimnser--Popa basis} instead of an \emph{orthonormal basis}. As we will show in Section \ref{sec: fusion hah!}, this can be done by modifying the consideration of \emph{reduced words} to be reduced ``relative" to a pair of subgraphs. This is similar to considering double-cosets relative to a pair of subgroups coming from subgraphs in a graph product of groups.  We  define this notion of relatively reduced words in this section. In order to later show they give something akin to a Pimsner--Popa basis and compute the fusion rules, we will also need to compute some conditional expectations coming from relatively reduced words, which we also do in this section. These formulas for conditional expectation will also be used to investigate \emph{relatively diffuseness} (i.e. lack of intertwining) in Section \ref{sec: diffuse}.

\subsection{\texorpdfstring{$\mathcal{G}$}{G}-reduced words}

\begin{defn}

We define the following kinds of operations on words in the alphabet $\cV$:
\begin{itemize}
    \item An \textbf{admissible swap} switches two consecutive letters $w_i$ and $w_{i+1}$ that are adjacent vertices in $\cG$.
    \item A \textbf{splitting} replaces one occurrence of a letter $w_i$ by two copies of $w_i$.  (For example, $1231$ could be transformed to $12231$ by splitting the second letter.)
    \item A \textbf{merge} replaces two consecutive occurrences of the same letter by one occurrence of the letter.
\end{itemize}
Two words are said to be \textbf{equivalent} if one can be transformed into the other by a sequence of these three types of operations. We denote this by $w \approx \hat{w}$. 
\end{defn}

It is easy to see that this is indeed an equivalence relation.  It is reflexive and transitive by construction.  It is symmetric because a swap operation is reversed by another swap, and the splitting and a merge operations are inverse to each other.  Moreover, every word is equivalent to some reduced word through a sequence of admissible swaps and merges (see \cite[Lemma 1.3(1)]{CaFi17}).

In the sequel, we will use the following characterization of when two reduced words are equivalent.

\begin{prop} \label{prop: reduced word equivalence}
Let $\cG = (\cV,\cE)$ be a graph.  Let $w = w_1 \dots w_m$ and $\hat{w} = \hat{w}_1 \dots \hat{w}_n$ be two words in the alphabet $\cV$. 
 Let $w = w_1 \dots w_m$ and $\hat{w} = \hat{w}_1 \dots \hat{w}_n$ be two $\cG$-reduced words.  Then the following are equivalent:
 \begin{enumerate}[(i)]
     \item $w$ and $\hat{w}$ are equivalent;\label{part:equivalent}
     
     \item $w$ can be transformed into $\hat{w}$ by a sequence of admissible swaps;\label{part:swaps}
     
     \item $m = n$ and there is a permutation $\sigma: [m] \to [m]$ such that
     \begin{itemize}
         \item $\hat{w}_{\sigma(i)} = w_i$;
         \item if $i < j$ and $w_i$ is not adjacent to $w_j$, then $\sigma(i) < \sigma(j)$.
     \end{itemize}
     \label{part:permutation}
 \end{enumerate}
\end{prop}

This proposition is a strengthening of \cite[Lemma 1.3]{CaFi17}.  For instance, \cite[Lemma 1.3]{CaFi17} showed that if $w$ and $\hat{w}$ are equivalent, then $m = n$ and there is some permutation matching the letters of $w$ and $\hat{w}$, but did not characterize the exact properties this permutation should have in order to get the reverse implication.  Moreover, they expressed condition (\ref{part:swaps}) as ``Type II equivalence'' and stopped short of showing it is the same as equivalence in the case of reduced words.

For the proof, (\ref{part:swaps}) $\implies$ (\ref{part:equivalent}) is immediate and (\ref{part:permutation}) $\implies$ (\ref{part:swaps}) follows by induction.  The implication (\ref{part:equivalent}) $\implies$ (\ref{part:permutation}) or (\ref{part:swaps}) is nontrivial since it involves reasoning about non-reduced words in intermediate stages of the sequence of transformations.  We first take care of (\ref{part:permutation}) $\implies$ (\ref{part:swaps}).

\begin{lem} \label{lem: permutation implies equivalence}
Let $\cG = (\cV,\cE)$ be a graph.  Let $w = w_1 \dots w_m$ and $\hat{w} = \hat{w}_1 \dots \hat{w}_n$ be two words in the alphabet $V$, and suppose $\sigma: [m] \to [m]$ is a permutation with $\hat{w}_{\sigma(i)} = w_i$ such that if $i < j$ and $w_i$ is not adjacent to $w_j$, then $\sigma(i) < \sigma(j)$.   Then $w$ and $\hat{w}$ are equivalent by swaps.
\end{lem}

\begin{proof}
We proceed by induction on $m$.
If $\sigma(1)=1$, then $\sigma$ restricts to a permutation of $\{2,\cdots,m-1\}$ and we can apply our inductive hyphothesis. 
Otherwise, $i=\sigma^{-1}(1)>1$, and $\hat{w}_{i}$ must be adjacent to $\hat{w}_{1},\cdots,\hat{w}_{i-1}$.  
Therefore, by successive swaps, we may move $w_{1}=\hat{w}_{i}$ to the left past $\hat{w}_1$, \dots, $\hat{w}_{i-1}$.  Then note that $\sigma$ restricts to a permutation of $m-1$ elements satisfying the original hypotheses for the words $w' = w_2 \dots w_m$ to $\hat{w}' = \hat{w}_1 \dots \hat{w}_{\sigma(1)-1} \hat{w}_{\sigma(1)+1} \dots \hat{w}_m$.  By the inductive hypothesis, $w'$ and $\hat{w}'$ are equivalent by a sequence of swaps, and hence $w$ and $\hat{w}$ are equivalent by a sequence of swaps as desired.
\end{proof}

For (\ref{part:equivalent}) $\implies$ (\ref{part:permutation}), we have to produce a permutation out of the sequence of operations.  It is easy to see that an admissible swap corresponds to a transposition permutation satisfying the monotonicity condition in (\ref{part:permutation}).  However, if we perform a split or a merge operation, then naturally two indices are mapped to one or vice versa, so in that setting, we need to replace the permutation (i.e.\ bijective function) by a relation from $[m]$ to $[n]$.

Recall that a \emph{relation} $R: A \to B$ between two sets $A$ and $B$ is a subset of $R \subseteq A \times B$. Given relations $R: A \to B$ and $S: B \to C$, the composition $S \circ R$ is defined by
    \[
        S \circ R = \{(a,c) \in A \times C: \text{ there exists } b \in B \text{ with } (a,b) \in R \text{ and } (b,c) \in S\}.
    \]
Note that this definition extends the composition of functions. 

\begin{defn} \label{def: monotone matching}
Let $\mathcal{G} = (\mathcal{V},\mathcal{E})$ be a graph.  Let $w = w_1 \dots w_m$ and $\hat{w} = \hat{w}_1 \dots \hat{w}_n$ be two words in the alphabet $V$.  A \textbf{$\mathcal{G}$-monotone matching} from $w$ to $\hat{w}$ is a relation $R: [m] \to [n]$ satisfying the following conditions:
\begin{enumerate}[(1)]
    \item For every $i \in [m]$, there is some $j \in [n]$ with $(i,j) \in R$.
    \item For every $j \in [n]$, there is some $i \in [m]$ with $(i,j) \in R$.
    \item If $(i,j) \in R$, then $w_i = \hat{w}_j$.
    \item If $(i,j) \in R$ and $(i',j') \in R$ and $w_i$ is not adjacent to $w_{i'}$ in $\cG$, then $i \leq i'$ iff $j \leq j'$.
\end{enumerate}
\end{defn}

Note in the case that the relation $R$ is a bijective function, then (1) and (2) of Definition \ref{def: monotone matching} hold, while (3) and (4) reduce to the conditions on the permutation $\sigma$ in Proposition~\ref{prop: reduced word equivalence}.(\ref{part:permutation}).

\begin{lem} \label{lem: monotone matching}
Let $\mathcal{G} = (\mathcal{V},\mathcal{E})$ be a graph.  Let $w = w_1 \dots w_m$ and $\hat{w} = \hat{w}_1 \dots \hat{w}_n$ be two words in the alphabet $V$.  If $w$ and $\hat{w}$ are equivalent, then there exists a $\cG$-monotone matching from $w$ to $\hat{w}$. 
\end{lem}

\begin{proof}
It suffices to show that (a) that each of the operations leads to a $\mathcal{G}$-monotone matching and (b) that a $\mathcal{G}$-monotone matching from $w$ and $\hat{w}$ and a $\mathcal{G}$-monotone matching from $\hat{w}$ to $\tilde{w}$ compose to form a $\mathcal{G}$-monotone matching from $w$ to $\tilde{w}$. For (a):
\begin{itemize}
    \item If $\hat{w}$ is obtained from $w$ by swapping $i$ and $i+1$, where $w_i$ and $w_{i+1}$ are adjacent, then a $\mathcal{G}$-monotone matching $R: [m] \to [m]$ is given by the relation $R = \{(j,j): j \neq i, i+1\} \cup \{(i,i+1),(i+1,i)\}$.
    \item If $\hat{w}$ is obtained from $w$ by merging $i$ and $i+1$, where $w_i = w_{i+1}$, then the $\cG$-monotone matching $R: [m] \to [m-1]$ is given by the relation $R = \{(1,1), \dots, (i,i)\} \cup \{ (i+1,i), \dots, (m,m-1)\}$.
    \item If $\hat{w}$ is obtained from $w$ by splitting the index $i$ into $i$ and $i+1$, then the $\cG$-monotone matching is given by $R = \{(1,1), \dots,(i,i)\} \cup \{(i,i+1),\dots,(m,m+1)\}$.
\end{itemize}
For (b), suppose $\tilde{w} = \tilde{w}_1, \dots, \tilde{w}_o$ is another word, suppose $R$ is a $\mathcal{G}$-monotone matching from $w$ to $\hat{w}$, and $S$ is a $\mathcal{G}$-monotone matching from $\hat{w}$ to $\tilde{w}$.  One can check that is $S \circ R$ is a $\mathcal{G}$-monotone matching from $w$ to $\tilde{w}$ by verifying each condition directly:
\begin{enumerate}[(1)]
    \item Given $i \in [m]$, there exists some $j \in [n]$ with $(i,j) \in R$, and then there exists some $k \in [o]$ with $(j,k) \in S$, and hence $(i,k) \in S \circ R$.
    \item The second condition is checked in the symmetrical way.
    \item If $(i,k) \in S \circ R$, then there exists some $j \in [n]$ with $(i,j) \in R$ and $(j,k) \in S$.  Hence, $w_i = \hat{w}_j = \tilde{w}_k$.
    \item Let $(i,k), (i',k') \in S \circ R$.  Suppose $w_i$ and $w_{i'}$ are not adjacent.  Pick $j$ and $j' \in [n]$ with $(i,j), (i',j') \in R$ and $(j,k), (j',k') \in S$. Note $\hat{w}_j = \hat{w}_{j'}$ by condition (3) for $R$.  Hence, $i < i'$ iff $j \leq j'$ iff $k \leq k'$ by condition (4) applied to $R$ and $S$.\qedhere
\end{enumerate}
\end{proof}

\begin{lem}\label{lem:monotone mathcing is a bijection in reduced case}
Let $\mathcal{G} = (\mathcal{V},\mathcal{E})$ be a graph.  Let $w = w_1 \dots w_m$ and $\hat{w} = \hat{w}_1 \dots \hat{w}_n$ be two reduced words in the alphabet $\mathcal{V}$.  If $w$ and $\hat{w}$ are equivalent, then $m = n$ and there is a permutation $\sigma: [m] \to [m]$ with $\hat{w}_{\sigma(i)} = w_i$ such that if $i < i'$ and $w_i$ is not adjacent to $w_{i'}$, then $\sigma(i) < \sigma(i')$.
\end{lem}

\begin{proof}
By the previous lemma, there exists a $\mathcal{G}$-monotone matching $R$ from $w$ to $\hat{w}$.  We claim that $R$ defines a bijection.

For each $i \in [m]$, we know that there exists some $j \in [n]$ with $(i,j) \in R$.  We claim that this is $j$ is unique.  Suppose that $(i,j) \in R$ and $(i,j') \in R$ with $j < j'$.  Since $\hat{w}$ is reduced, there exists some $\ell$ strictly between $j$ and $j'$ such that $\hat{w}_\ell$ is not equal or adjacent to $\hat{w}_j$.  Moreover, there exists some $k \in [m]$ with $(k,\ell) \in R$.  Then condition (4) of $\mathcal{G}$-monotonicity tells us that $j \leq \ell \leq j'$ implies that $i \leq k \leq i$, hence $k = i$.  However, this contradicts that $w_i=\hat{w}_{j} \neq \hat{w}_{\ell}=w_k$.

A symmetrical argument shows that for every $j \in [n]$, there is a unique $i \in [m]$ with $(i,j) \in R$.  Thus, $R$ defines a bijection as desired, so that $m = n$ and $R$ has the form $R = \{(i,\sigma(i): i \in [m] \}$ for some permutation $\sigma$.  By Definition \ref{def: monotone matching}, we see that if $i < i'$ and $w_i$ is not adjacent to $w_{i'}$, then $\sigma(i) < \sigma(i')$.
\end{proof}

This lemma completes the proof of (\ref{part:equivalent}) $\implies$ (\ref{part:permutation}) in Proposition \ref{prop: reduced word equivalence}. 
\begin{remark}
If $w$ and $\hat{w}$ are equivalent $\mathcal{G}$-reduced words, note that the permutation $\sigma$ is uniquely determined by the property that for each $v \in V$, $\sigma$ maps $\{i: w_i = v\}$ onto $\{j: \hat{w}_j = v\}$ monotonically.  In particular, the permutation in Proposition \ref{prop: reduced word equivalence}.(\ref{part:permutation}) is unique.
\end{remark}

\begin{remark}
The method of proof more generally shows that arbitrary words $w$ and $\hat{w}$ are equivalent if and only if there exists a $\mathcal{G}$-monotone matching from $w$ to $\hat{w}$.  Indeed, Lemma \ref{lem: monotone matching} shows that equivalence of $w$ and $\hat{w}$ implies the existence of a $\mathcal{G}$-monotone matching.  On the other hand, suppose there is a $\mathcal{G}$-monotone matching from $w$ to $\hat{w}$.  Note $w$ and $\hat{w}$ are equivalent to some reduced words $w'$ and $\hat{w}'$, and hence there are $\mathcal{G}$-monotone matchings from $w'$ to $w$, from $w$ to $\hat{w}$, and from $\hat{w}$ to $\hat{w}'$.  The composition yields a $\cG$-monotone matching from $w'$ to $\hat{w}'$, so by Proposition \ref{prop: reduced word equivalence}, $w'$ and $\hat{w}'$ are equivalent by swaps, hence also $w$ and $\hat{w}$ are equivalent.
\end{remark}

\subsection{Relatively \texorpdfstring{$\mathcal{G}$}{G}-reduced words}

In order to compute conditional expectations and study relative properties of subalgebras, we use a relative notion of reduced word.

\begin{defn}
Let $\mathcal{G}=(\mathcal{V},\mathcal{E})$ be a graph and $V_1$, $V_2 \subseteq \mathcal{V}$.  Let $w$ be a word in the alphabet $\mathcal{V}$.
\begin{enumerate}[(1)]
    \item $w$ is \textbf{$\mathcal{G}$-reduced relative to $V_1$ on the left} if $v_1w$ is $\mathcal{G}$-reduced for every letter $v_1 \in V_1$.
    
    \item $w$ is \textbf{$\mathcal{G}$-reduced relative to $V_2$ on the right} if $wv_2$ is $\cG$-reduced for every $v_2 \in V_2$.
    
    \item $w$ is \textbf{$\mathcal{G}$-reduced relative to $(V_1,V_2)$} if both (1) and (2) hold.
\end{enumerate}
\end{defn}

\begin{remark}
In the case of $\varnothing\subset \cV$, we take $w$ being $\cG$-reduced relative to $\varnothing$ on the left or right to just mean that $w$ is $\cG$-reduced.  Consequently, $w$ is $\mathcal{G}$-reduced relative to $V_1$ on the left if and only if $w$ is $\mathcal{G}$-reduced relative to $(V_1,\varnothing)$.  Similarly, $w$ is $\mathcal{G}$-reduced relative to $V_2$ on the right if and only if $w$ is $\mathcal{G}$-reduced relative to $(\varnothing,V_2)$.
We also note that all relatively $\cG$-reduced words are, in particular, $\cG$-reduced words.
\end{remark}

The next three lemmas show existence and uniqueness of a certain factorization of reduced words based on the vertex sets $V_1$ and $V_2$. This will be useful in Section \ref{sec: fusion hah!} when we compute the fusion rules for bimodules arising from subgraphs. 

\begin{lem} \label{lem: join with relatively reduced}
Let $\mathcal{G}=(\cV,\cE)$ be a graph and $V_1$, $V_2 \subseteq \cV$.  Suppose that $w = w^{(1)} \cdot w^{(2)} \cdot w^{(3)}$ where
\begin{enumerate}[(1)]
    \item $w^{(1)}$ is a $\mathcal{G}$-reduced word in  the alphabet $V_1$.
    \item $w^{(2)}$ is $\mathcal{G}$-reduced relative to $(V_1,V_2)$.
    \item $w^{(3)}$ is a word in the alphabet $V_2$ that is $\mathcal{G}$-reduced relative to $(U,\varnothing)$, where $U$ is the set of vertices in $V_1 \cap V_2$ that are adjacent to \emph{all} the letters in $w^{(2)}$.
\end{enumerate}
Then $w$ is $\mathcal{G}$-reduced.
\end{lem}

\begin{proof}
Denote $w=w_1\cdots w_n$ and suppose that $i < j$ with $w_i = w_j$.  We must find some $i<k<j$ such that $w_k$ is not adjacent to $w_i = w_j$.  We proceed in cases:
\begin{enumerate}[(A)]
    \item If $w_i$ and $w_j$ are both from $w^{(1)}$, the claim follows because $w^{(1)}$ is reduced.  Similarly for $w^{(2)}$ and $w^{(3)}$.
    
    \item Suppose that $w_i$ comes from $w^{(1)}$ and $w_j$ comes from $w^{(2)}$.  Because $w_i \in V_1$ and $w^{(2)}$ is $\cG$-reduced relative to $(V_1,V_2)$, the word $w_i \cdot w^{(2)}$ is $\mathcal{G}$-reduced, and hence there exists some index $k < j$, within $w^{(2)}$, such that $w_k$ is not equal or adjacent to $w_i = w_j$.
    
    \item Suppose that $w_i$ comes from $w^{(2)}$ and $w_j$ comes from $w^{(3)}$.  Using that $w_j \in V_2$ and thus $w^{(2)} \cdot w_j$ is $\mathcal{G}$-reduced, we can argue analogously to the previous case.
    
    \item Finally, suppose that $w_i$ is from $w^{(1)}$ and $w_j$ is from $w^{(3)}$.  Note that $w_i = w_j$ must be in $V_1 \cap V_2$.  Then there are two subcases: (a) $w_i\not\in U$; and (b) $w_i\in U$. For (a), the definition of $U$ implies there exists some index $k$ from $w^{(2)}$ such that $w_k$ is not adjacent to $w_i$.  Since $k$ is from $w^{(2)}$, we have $i < k < j$, so we are done.  For (b), because $w^{(3)}$ is $\mathcal{G}$-reduced relative to $(U,\varnothing)$, we know $w_i \cdot w^{(3)}$ is $\mathcal{G}$-reduced, and so there is some index $k < j$ from $w^{(3)}$ such that $w_k$ is not adjacent to $w_i$.\qedhere
\end{enumerate}
\end{proof}

\begin{lem} \label{lem: word relative decomposition}
Let $\mathcal{G}=(\mathcal{V},\mathcal{E})$ be a graph and $V_1$, $V_2 \subseteq \mathcal{V}$.  Every word $w$ is equivalent to a word of the form $w^{(1)} \cdot w^{(2)} \cdot w^{(3)}$ satisfying the conditions in Lemma \ref{lem: join with relatively reduced}.
\end{lem}

\begin{proof}
Since every word is equivalent to a $\cG$-reduced word, we may assume without loss of generality that $w$ is $\cG$-reduced.
Let $a(w) = \min\{j: w_j \not \in V_1\}$ and let $b(w) = \max\{j: w_j \not \in V_2\}$.  Let $C$ be the set of $\cG$-reduced words equivalent to $w$.  Let $w'$ be an element in $C$ that maximizes $a(w')$.  Then $w''$ be an element in $C$ that minimizes $b(w'')$ subject to the constraint that $a(w'') = a(w')$.  Write $a = a(w') = a(w'')$ and $b = b(w'')$.  Write $w'' = w^{(1)} \cdot w^{(2)} \cdot w^{(3)}$ where
\begin{align*}
    w^{(1)} &= w_1'' \dots w_{a-1}'' \\
    w^{(2)} &= w_a'' \dots w_b'' \\
    w^{(3)} &= w_{b-1}'' \dots w_\ell'',
\end{align*}
where $\ell$ is the length of $w''$. Note that $w^{(1)}$ is a word in the alphabet in $V_1$ and $w^{(3)}$ is word in the alphabet $V_2$. Moreover, $w^{(1)}$, $w^{(2)}$, and $w^{(3)}$ are all $\mathcal{G}$-reduced since they are subwords of the $\mathcal{G}$-reduced word $w''$.  We will complete the proof via a series of claims.\\

\noindent\textbf{Claim 1:} $w^{(2)} \cdot w^{(3)}$ is $\mathcal{G}$-reduced relative to $(V_1,\varnothing)$.

\noindent Fix $v \in V_1$ and let $i < j$ be two indices in $v \cdot w^{(2)} \cdot w^{(3)}$ labeled with the same vertex. We must show there is some index in between labeled by a non-adjacent vertex. If the two indices $i$ and $j$ are both from $w^{(2)} \cdot w^{(3)}$, then it suffices to note that is $w^{(2)} \cdot w^{(3)}$ is $\mathcal{G}$-reduced since it is a subword of $w''$, which is reduced because it is equivalent by admissible swaps to $w$. Otherwise, $i$ corresponds to the first letter $v$ in $v \cdot w^{(2)} \cdot w^{(3)}$.  Suppose for contradiction that there does not exist some index $k$ between $i$ and $j$ such that $w_k''$ is not adjacent to $w_j''$.  Then all the letters between $v$ and $w_j''$ in $v \cdot w^{(2)} \cdot w^{(3)}$ are adjacent to $v$, and hence $w_j''$ can be moved past them to the left by repeated swaps, so that it comes to the left side of $w^{(2)} \cdot w^{(3)}$.  Thus, by grouping the letter $w_j''$ with $w^{(1)}$ instead of $w^{(2)} \cdot w^{(3)}$, we obtain a contradiction to the assumption that $a(w'')$ is maximal.\\

\noindent\textbf{Claim 2:} $w^{(2)}$ is $\mathcal{G}$-reduced relative to $(\varnothing,V_2)$.

\noindent Fix $v \in V_2$ and let $i<j$ be two indices in $w^{(2)}\cdot v$ labeled by the same vertex.
 Note that $w^{(2)}$ is $\mathcal{G}$-reduced, so if $i$ and $j$ are both from $w^{(2)}$ then we are done. Otherwise, $j$ corresponds to the last letter $v$ in $w^{(2)}$.  If $w_k''$ is adjacent to $w_i''$ for all $k>i$ among the indices of $w^{(2)}$, then arguing as in Claim 1 we would contradict minimality of $b(w'')$.\\

Observe that the combination of Claims 1 and 2 give that $w^{(2)}$ is $\cG$-reduced relative to $(V_1,V_2)$. In the final claim, let $U$ be the set of vertices in $V_1\cap V_2$ that are adjacent to all the letters in $w^{(2)}$.\\

 \noindent\textbf{Claim 3:} $w^{(3)}$ is $\mathcal{G}$-reduced relative to $(U,\varnothing)$.

 \noindent Fix $v \in U$ and let $i<j$ be two indices in $v\cdot w^{(3)}$. labeled by the same vertex.  Since $w^{(3)}$ is $\mathcal{G}$-reduced, if $i$ and $j$ are both from $w^{(3)}$ then we are done. So suppose $i$ corresponds to $v$ and that there is some index $j$ in $w^{(3)}$ with $w_j'' = v$. Since $w^{(2)}\cdot w^{(3)}$ is $\mathcal{G}$-reduced relative to $(V_1,\varnothing)$ by Claim 1 and $U \subseteq V_1$, there must be some index $k < j$ in $w^{(2)} \cdot w^{(3)}$ with $w_k''$ not adjacent to $v = w_j''$.  By definition of $U$, $v$ is adjacent to all the letters in $w^{(2)}$.  Hence, the index $k$ must have come from $w^{(3)}$.  Thus, $w_k$ occurs as a letter in $v \cdot w^{(3)}$ between $v$ and $w_j''$, and $w_k''$ is not adjacent to $v$.
\end{proof}

\begin{lem} \label{lem: word relative decomposition 2}
Let $\mathcal{G}=(\mathcal{V},\mathcal{E})$ be a graph and $V_1$, $V_2 \subseteq \mathcal{V}$.  Let $w = w^{(1)} \cdot w^{(2)} \cdot w^{(3)}$ and $\hat{w} = \hat{w}^{(1)} \cdot \hat{w}^{(2)} \cdot \hat{w}^{(3)}$ satisfy the conditions in Lemma \ref{lem: join with relatively reduced} (here in the third condition for $w$ and $\hat{w}$, we use respectively $U$ and $\hat{U}$, where $U$ and $\hat{U}$ are the sets of vertices in $V_1 \cap V_2$ that are adjacent to all letters in $w^{(2)}$ and $\hat{w}^{(2)}$ respectively).  If $w \approx \hat{w}$, then $w^{(1)} \approx \hat{w}^{(1)}$, $w^{(2)} \approx \hat{w}^{(2)}$, and $w^{(3)} \approx \hat{w}^{(3)}$.
\end{lem}

\begin{proof}
First observe that  $w^{(2)} \cdot w^{(3)}$ and $\hat{w}^{(2)} \cdot \hat{w}^{(3)}$ are both $\mathcal{G}$-reduced relative to $(V_1,\varnothing)$, by applying Lemma~\ref{lem: join with relatively reduced} to $v\cdot w^{(2)} \cdot w^{(3)}$ and  $\hat{w}^{(2)} \cdot \hat{w}^{(3)}$ to $v\in V_1$.

Now, since $w \approx \hat{w}$, Proposition \ref{prop: reduced word equivalence} shows that there is a permutation $\sigma$ with $w_i = \hat{w}_{\sigma(i)}$, such that if $i < j$ and $w_i$ is not adjacent to $w_j$, then $\sigma(i) < \sigma(j)$. We claim that $\sigma$ maps the indices of $w^{(2)} \cdot w^{(3)}$ into the letters of $\hat{w}^{(2)} \cdot \hat{w}^{(3)}$.  We proceed by induction on the indices of $w^{(2)} \cdot w^{(3)}$, from left to right.  Let $i$ be one of these indices and suppose the claim is known for all indices to its left in $w^{(2)} \cdot w^{(3)}$.  There are now two cases:
\begin{itemize}
    \item Suppose $w_i \not \in V_1$.  Then $\hat{w}_{\sigma(i)} = w_i$ is not in $V_1$ and hence $\sigma(i)$ cannot be one of the indices in $\hat{w}^{(1)}$, so it must be one of the indices in $\hat{w}^{(2)} \cdot \hat{w}^{(3)}$.
    \item Suppose that $w_i \in V_1$.  Then because $w^{(2)} \cdot w^{(3)}$ is $\mathcal{G}$-reduced relative to $(V_1,\varnothing)$, we know $w_i \cdot w^{(2)} \cdot w^{(3)}$ is $\mathcal{G}$-reduced, so there must exist some index $j < i$ in $w^{(2)} \cdot w^{(3)}$ such that $w_j$ is not adjacent to $w_i$ in $\mathcal{G}$.  By induction hypothesis, $\sigma(j)$ is one of the indices in $\hat{w}^{(2)} \cdot \hat{w}^{(3)}$.  By Lemma~\ref{lem: monotone matching}, we must have $\sigma(i) > \sigma(j)$ and hence $\sigma(i)$ is one of the indices in $\hat{w}^{(2)} \cdot \hat{w}^{(3)}$, as desired.
\end{itemize}
By symmetrical reasoning, $\sigma^{-1}$ must map the indices of $\hat{w}^{(2)} \cdot \hat{w}^{(3)}$ into the indices of $w^{(2)} \cdot w^{(3)}$.  Therefore, $\sigma$ restricts to $\cG$-monotone matchings from $w^{(1)}$ to $\hat{w}^{(1)}$ and from $w^{(2)} \cdot w^{(3)}$ to $\hat{w}^{(2)} \cdot \hat{w}^{(3)}$. That is, $w^{(1)} \approx \hat{w}^{(1)}$ and $w^{(2)}\cdot w^{(3)} \approx \hat{w}^{(2)} \cdot \hat{w}^{(3)}$.

Finally, we argue $\sigma$ as above maps the indices of $w^{(2)}$ into the indices of $\hat{w}^{(2)}$. We again proceed by induction on the indices of $w^{(2)}$, this time from right to left.  Let $i$ be an index in $w^{(2)}$ and the claim is already known for all indices $j$ to its right.  We again have two cases:
\begin{itemize}
    \item If $w_i \not \in V_2$, then $\sigma(i)$ must be an index in $\hat{w}^{(2)}$.
    \item If $w_i \in V_2$, then since $w^{(2)}$ is $\mathcal{G}$-reduced relative to $(V_1,V_2)$, then there is some index $j > i$ in $w^{(2)}$ such that $w_j$ is not adjacent to $w_i$.  Then $\sigma(i) < \sigma(j)$, which is by induction hypothesis an index in $\hat{w}^{(2)}$.  Thus, $\sigma(i)$ is an index in $\hat{w}^{(2)}$.
\end{itemize}
Symmetrically, $\sigma^{-1}$ maps the indices of $\hat{w}^{(2)}$ into the indices of $w^{(2)}$. Thus, as above, we have $w^{(2)}\approx \hat{w}^{(2)}$ and $w^{(3)} \approx \hat{w}^{(3)}$.
\end{proof}

\subsection{Computation of conditional expectation}\label{sec: cond expec formulas}

We recall the following facts which follow from the Fock space description of $L^{2}$ of the graph product in \cite[Section 2.1]{CaFi17}.

\begin{lem}[Remark 2.7 of \cite{CaFi17}] \label{lem: reduced words span}
Let $\mathcal{G}=(\mathcal{V},\mathcal{E})$ be a graph,  let $\{(M_v,\varphi_v):v \in \cV\}$ be a family of stacial von Neumann algebras, and $(M,\varphi) = \gp_{v \in \mathcal{G}} (M_v,\varphi_v)$.  The $*$-subalgebra generated by $(M_v)_{v \in \mathcal{V}}$ is spanned by $1$ and elements of the form $x_1 \dots x_m$ where $x_j \in M_{w_j}$ with $\varphi(x_j)=0$ for some $\cG$-reduced word $w = w_1 \dots w_m$.
\end{lem}

\begin{lem}[Comments following Remark 2.11 of \cite{CaFi17}] \label{lem: reduced words orthogonal}
Let $\mathcal{G}=(\mathcal{V},\mathcal{E})$ be a graph and $(M,\varphi) = \gp_{v \in \mathcal{G}} (M_v,\varphi_v)$.  Let $w = w_1 \dots w_m$ and $\tilde{w} = \tilde{w}_1 \dots \tilde{w}_n$ be $\mathcal{G}$-reduced words.  Let $x_j \in M_{w_j} \cap \ker(\varphi)$ and $\tilde{x}_j \in M_{\tilde{w}_j} \cap \ker(\varphi)$.
\begin{enumerate}[(i)]
    \item If $w$ and $\tilde{w}$ are not equivalent, then $\varphi((x_1 \dots x_m)^* (\tilde{x}_1 \dots \tilde{x}_n)) = 0$.
    
    \item If $w$ and $\tilde{w}$ are equivalent, then
        \[
            \varphi((x_1 \dots x_m)^*(\tilde{x}_1 \dots \tilde{x}_m)) = \varphi(x_1^{*} \tilde{x}_{\sigma(1)}) \dots \varphi(x_m^{*} \tilde{x}_{\sigma(m)}^*),
        \]
    where the permutation $\sigma: [m] \to [m]$ is the $\mathcal{G}$-monotone matching from $w$ to $\tilde{w}$ guaranteed by Lemma~\ref{lem:monotone mathcing is a bijection in reduced case}, which also gives $m=n$.
    
\end{enumerate}
\end{lem}

\begin{lem}[Remark 2.14 of \cite{CaFi17}] \label{lem: cond exp reduced words CF}
Let $\mathcal{G}=(\mathcal{V},\mathcal{E})$ be a graph,  let $\{(M_v,\varphi_v):v \in \cV\}$ be a family of stacial von Neumann algebras, let $(M,\varphi) = \gp_{v \in \mathcal{G}} (M_v,\varphi_v)$, and let $V_0 \subseteq \mathcal{V}$.  For a $\cG$-reduced word $w = w_1 \dots w_m$, if $x_j \in M_{w_j}$ for each $j = 1$, \dots, $m$ then $E_{M_{V_0}}[x_1 \dots x_m] = 0$ unless $w_1,\ldots, w_n\in V_0$.
\end{lem}

Our goal is to prove a conditional analog of Lemma \ref{lem: reduced words orthogonal}.

\begin{lem} \label{lem: Pimnser Popa computation}
Let $\mathcal{G}=(\mathcal{V},\mathcal{E})$ be a graph, let $\{(M_v,\varphi_v):v \in \cV\}$ be a family of stacial von Neumann algebras, and $V_1, V_2 \subseteq \mathcal{V}$. Let $w = w_1 \dots w_m$ and $\tilde{w} = \tilde{w}_1 \dots \tilde{w}_n$ be $\mathcal{G}$-reduced words relative to $(V_1,V_2)$.  Let $x_j \in M_{w_j} \cap \ker(\varphi)$ and $\tilde{x}_j \in M_{\tilde{w}_j} \cap \ker(\varphi)$, and write
\[
x = x_1 \dots x_m, \qquad \tilde{x} = \tilde{x}_1 \dots \tilde{x}_n.
\]
(In the case that $w$ or $\tilde{w}$ are empty, that is, $m = 0$ or $n = 0$, we take by convention $x = 1$ or $\tilde{x} =1$ respectively.)  Let $U$ be the set of vertices in $V_1 \cap V_2$ that are adjacent to \emph{all} letters of $w$ (note: if $w$ is the empty word, then $U=V_{1}\cap V_{2}$, by convention).  Then
    \begin{equation} \label{eq: conditional expectation formula}
        E_{M_{V_{2}}}(\tilde{x}^{*}yx)=\varphi(\tilde{x}^{*}x)E_{M_{U}}(y), 
        \qquad \forall y\in M_{V_{1}}.
    \end{equation}
In particular, $E_{M_{V_{2}}}(\tilde{x}^{*}yx)=0$ for all $y\in M_{V_1}$ if $w$ and $\tilde{w}$ are not equivalent.
\end{lem}

\begin{proof}
It suffices to show that for all $y\in M_{V_1}$ and $z \in M_{V_2}$, we have
\begin{equation}\label{eqn: inner prod form for con exp}
\varphi(\tilde{x}^{*}yxz) = \varphi(\tilde{x}^{*}x) \varphi(E_{M_{U}}(y)z).
\end{equation}
By Lemma \ref{lem: reduced words span}, it further suffices to prove the claim when
\begin{equation} \label{eq: def of z}
z = z_1 \dots z_\ell, \quad z_j \in M_{a_j} \cap \ker(\varphi),
\end{equation}
where $a = a_1 \dots a_\ell$ is a $\mathcal{G}$-reduced word in the alphabet $V_2$.  Additionally, by Proposition~\ref{prop: reduced word equivalence}.\ref{part:swaps} and Lemma~\ref{lem: word relative decomposition}, we can assume without loss of generality that $a = a^{(1)}\cdot a^{(2)}$ where $a^{(1)}$ is a $\mathcal{G}$-reduced word in $U$ and $a^{(2)}$ is a $\mathcal{G}$-reduced relative to $(U,\varnothing)$.  This results in a corresponding factorization $z = z^{(1)} z^{(2)}$ with $z^{(1)} \in M_U$.  Then
\[
\varphi(\tilde{x}^{*}yxz^{(1)}z^{(2)}) = \varphi(\tilde{x}^*yz^{(1)}xz^{(2)}) \qquad  
\]
Thus, it suffices to prove the claim with $y$ replaced by $yz^{(1)}$ and $z$ replaced by $z^{(2)}$.

In other words, we can assume without loss of generality that $z$ is given by \eqref{eq: def of z} where $a$ is $\mathcal{G}$-reduced relative to $(U,\varnothing)$.  Furthermore, again by Lemma \ref{lem: reduced words span}, it suffices to consider the case where
\[
y = y_1 \dots y_k, \qquad y_j \in M_{b_j} \cap \ker(\varphi),
\]
where $b = b_1 \dots b_k$ is a $\mathcal{G}$-reduced word in $V_1$.  By Lemma \ref{lem: join with relatively reduced}, $b \cdot w \cdot a$ is $\mathcal{G}$-reduced.  Moreover, by Lemma \ref{lem: word relative decomposition 2}, the only way for $\tilde{w}$ and $b \cdot w \cdot a$ to be equivalent is if $\tilde{w} \approx w$ and $\varnothing \approx b$ and $\varnothing \approx a$ (hence $a$ and $b$ are empty).  Similarly, the only way for $\tilde{w}$ and $b \cdot w$ to be equivalent is if $w \approx \tilde{w}$ and $b = \varnothing$.  Thus, the claim can be checked in several cases:
\begin{itemize}
    \item In the case $a = b = \varnothing$, so then $y = z = 1$, we have
    \[
    \varphi(\tilde{x}^*yxz) = \varphi(\tilde{x}^*x) = \varphi(\tilde{x}^*x) \varphi(E_{M_{U}}[1]1) = \varphi(\tilde{x}^*x) \varphi(E_{M_{U}}[y]z).
    \]
    \item In case $a = \varnothing$ and $b \neq \varnothing$, then since $b \cdot w$ is not equivalent to $\tilde{w}$, we get $\varphi(\tilde{x}^*yx) = 0$ by Lemma \ref{lem: reduced words orthogonal},  hence the left-hand side of \eqref{eq: conditional expectation formula} is zero.  Meanwhile, $\varphi(y) = 0$ by definition of the graph product, so the right-hand side of \eqref{eq: conditional expectation formula} is $\varphi(E_{M_{U}}[y]) = \varphi(y) = 0$.
    \item In case $a \neq \varnothing$, then again $b \cdot w \cdot a$ is not equivalent to $\tilde{w}$, and hence the left-hand side of (\ref{eqn: inner prod form for con exp}) is zero, by Lemma \ref{lem: reduced words orthogonal}.  Meanwhile, since the word $a$ is $\mathcal{G}$-reduced relative to $(U,\varnothing)$, the element $z$ is orthogonal $M_U$ by Lemma \ref{lem: cond exp reduced words CF}, hence $\varphi(E_{M_{U}}[y]z) = 0$, so the right-hand side of (\ref{eqn: inner prod form for con exp}) is zero. \qedhere
\end{itemize}
\end{proof}

\section{Non-intertwining}\label{sec: diffuse}

Let $(M,\tau)$ be a tracial von Neumann algebra and $B,N\leq M$. We say that \textbf{$B$ intertwines into $N$ inside of $M$} if there exist nonzero projections $p_0 \in B$, $q_0 \in N$, and a normal unital $*$-homomorphism $\theta: p_0Pp_0 \to q_0Qq_0$, together with a nonzero partial isometry $v\in q_0Mp_0$ such that $v^{*}v=p_{0},$ $vv^{*}=q_{0}$ and $\theta(x)v = vx$ for all $x \in p_0Pp_0$. In this case one writes $N \preceq_M B$.

\begin{thm}[{\cite[Section 2]{PopaStrongRigidity}}] \label{T:Popa IBBT}
Let $(M, \tau)$ be a tracial von Neumann algebra, $p, q \in \mathcal{P}(M)$ projections, and $B\leq pMp$, $N\leq qMq$. Then the following are equivalent:
\begin{enumerate}[(i)]
\item $N\not\preceq_{M}B$;

\item there is a net $(u_n)_{n \in I}$ in $\mathcal{U}(N)$ with $\|E_B(xu_ny)\|_2 \to 0 \text{ for all } x, y \in M$;\label{I:total mixing over Q interwine}

\item for any subgroup $G\leq \mathcal{U}(N)$ with $N = W^{*}(G)$ there is a net $(u_n)_{n \in I}$ in $G$ satisfying $\|E_B(xu_ny)\|_2 \to 0 \text{ for all } x, y \in M$; \label{I:mixing over Q}

\item any $ P$-$Q$-sub-bimodule $K$ of $pL^2(M)q $ satisfies $\operatorname{dim}(K_Q)=+\infty $.\label{I:FD bimodule}
\end{enumerate}
\end{thm}

In this section we completely characterize when two subalgebras corresponding to induced subgraphs do not intertwine into each other (partial results were previously obtained in \cite[Lemma 2.27]{CaFi17}).
We will  say ``$N$ is diffuse relative to $B$ in $M$" to mean $N\not\preceq_{M} B$. This is motivated by the case $B=\bC$, since $N\not\preceq_{M} \bC$ means precisely that $N$ is diffuse.  
This will also provides intuition for our main result in this section, since in our setting $N$ being diffuse relative to $B$ in $M$ will be equivalent to a combination of conditions which either require that a vertex algebra is diffuse or a lack of edges between subgraphs (i.e. some ``free independence outside of the subgraph") in a manner analogous to Theorem~\ref{thm: char diffuse factor full intro}.(\ref{item: diffuse characterization intro}).

\begin{prop} \label{prop: relative diffuseness}
    Suppose that $\mathcal{G} = (\mathcal{V}, \mathcal{E})$ is a graph, and for each $v \in V$ let $(M_v, \tau_v)$ be a tracial von Neumann algebra such that $M_v$ contains a trace zero unitary.  Let $(M,\tau) = \gp_{v \in V}(M_v, \tau_v)$. For $V_1, V_2 \subseteq \mathcal{V}$, the following are equivalent:
    \begin{enumerate}[(i)]
    \item $M_{V_1}$ is diffuse relative to $M_{V_2}$ in $M$; \label{item: diffuse rel char}
    
    \item $M_{V_1}$ is diffuse relative to $M_{V_1\cap V_2}$ in $M$; \label{item: diffuse rel pass to subalg}
    
    \item at least one of the following holds: \label{item: diffuse rel sub alg char}
        \begin{enumerate}[(a)]
        \item there are $v \in V_1 \setminus V_2$ and $v' \in V_1 \cap V_2$ with $v\not\sim v' $; or \label{item: diffuse rel intersection edge}
        
        \item $M_{V_1 \setminus V_2}$ is diffuse; \label{item: diffuse rel comp diffuse}
        \end{enumerate}
        
    \item at least one of the following holds: \label{item: diffuse rel local char}
        \begin{enumerate}[(a)]
        \item there are $v \in V_1 \setminus V_2$ and $v' \in V_1$ with $v \neq v'$ and $v\not\sim v'$; or \label{item: not adj vertices give diffuseness}
        
        \item there is a $v \in V_1 \setminus V_2$ for which $M_{v}$ is diffuse. \label{item: diffuse rel implied by diffuse subalg}
        \end{enumerate}
    \end{enumerate}
\end{prop}

\begin{proof}

(\ref{item: diffuse rel char}) $\implies$ (\ref{item: diffuse rel pass to subalg}):
Using the characterization from Theorem~\ref{T:Popa IBBT}.(\ref{I:total mixing over Q interwine}), this follows from the identity $E_A=E_{A}\circ E_{B}$ for von Neumann subalgebras $A\subset B\subset M$ and the fact that the trace-preserving conditional expectation is contractive with respect to the $L^2$ norm.\\

\noindent(\ref{item: diffuse rel pass to subalg}) $\implies$ (\ref{item: diffuse rel sub alg char}):
We proceed by contrapositive and assume (\ref{item: diffuse rel intersection edge}) and (\ref{item: diffuse rel comp diffuse}) are false. It follows that
    \[
        M_{V_1} = M_{V_1\setminus V_2} \bar\otimes M_{V_1\cap V_2},
    \]
and $M_{V_1\setminus V_2}$ is not diffuse. Let $z\in M_{V_1\setminus V_2}$ be a central projection such that $zM_{V_1\setminus V_2}\cong M_d(\bC)$ for some $d\in \bN$. Suppose
    \[
        u=(u_{i,j})_{i,j=1}^d \in M_d( M_{V_1\cap V_2}) \cong (z\otimes 1)M_{V_1}
    \]
is a unitary. Observe that
    \begin{align*}
        1 = \frac{1}{d}\sum_{i,j=1}^d \|u_{i,j}\|_2^2 = \frac{1}{d} \sum_{i,j=1}^d \| E_{M_{V_1\cap V_2}}( e_{ii}u e_{ji})\|_2^2.
    \end{align*}
Hence $(z\otimes 1)M_{V_1}(z\otimes 1)\preceq_{(z\otimes 1)M(z\otimes 1)} M_{V_1\cap V_2}$, and consequently $M_{V_1} \preceq_M M_{V_1\cap V_2}$.\\    

\noindent(\ref{item: diffuse rel sub alg char}) $\implies$ (\ref{item: diffuse rel local char}): We again proceed  by contrapositive and assume (\ref{item: not adj vertices give diffuseness}) and (\ref{item: diffuse rel implied by diffuse subalg}) are false.  Then every $v \in V_1 \setminus V_2$ must be adjacent to every $v' \in V_1$, so in particular (\ref{item: diffuse rel intersection edge}) fails.  Moreover, any two vertices in $V_1 \setminus V_2$ are adjacent, that is, $V_1 \setminus V_2$ is a complete graph.  Since (\ref{item: diffuse rel implied by diffuse subalg}) fails, we know that for every $v \in V_1 \setminus V_2$ there is a minimal projection $p_v$ in $M_v$.  In particular, $p = \bigotimes_{v \in V_1 \setminus V_2} p_v$ is a minimal projection in $M_{V_1 \setminus V_2}$, and thus $M_{V_1 \setminus V_2}$ is not diffuse and hence (\ref{item: diffuse rel comp diffuse}) fails.\\

\noindent(\ref{item: not adj vertices give diffuseness}) $\implies$ (\ref{item: diffuse rel char}) Let $v, v'$ be as in (\ref{item: not adj vertices give diffuseness}).  Let $u_0$ be a trace zero unitary in $M_{v'}$ and let $u_1$ be a trace zero unitary in $M_v$.  We claim that for $x, y \in M$, we have
\begin{equation} \label{eq: diffuseness claim 1}
\lim_{k \to \infty} \norm{E_{M_{V_2}}[x(u_0u_1)^k y]}_2 = 0.
\end{equation}
It suffices to show this for a set of $x$ and $y$ that have dense linear span.  Hence, by Lemma \ref{lem: reduced words span}, we may assume that $x = x_1 \dots x_m$ with $x_j \in M_{v_j}$ for a $\cG$-reduced word $w_1 \dots w_m$ and with $\varphi(x_j) = 0$ (in the case that $x = 1$, we take $w$ to be the empty word).  Similarly, assume that $y = y_1 \dots y_n$ with $\varphi(y_j) = 0$ and $y_j \in M_{\hat{w}_j}$ with $\hat{w}$ a reduced word.

By Lemma \ref{lem: word relative decomposition}, $w$ is equivalent to $w^{(1)} \cdot w^{(2)} \cdot w^{(3)}$ where $w^{(1)}$ is a $\mathcal{G}$-reduced word in $V_2$, $w^{(3)}$ is $\mathcal{G}$-reduced word in $\{v,v'\}$, and $w^{(2)}$ is a $\mathcal{G}$-reduced word relative to $(V_2,\{v,v'\})$.  By swapping the $x_j$'s according to the swaps to transform $w$ into $w^{(1)} \cdot w^{(2)} \cdot w^{(3)}$, we then obtain a factorization $x = x^{(1)} x^{(2)} x^{(3)}$ where $x^{(j)}$ is a product of centered elements indexed by the word $w^{(j)}$. Similarly, $\hat{w}$ is equivalent to $\hat{w}^{(1)} \cdot \hat{w}^{(2)} \cdot \hat{w}^{(3)}$ where $\hat{w}^{(1)}$ is a reduced word in $\{v,v'\}$, $\hat{w}^{(3)}$ is a reduced word in $V_2$, and $\hat{w}^{(2)}$ is $\mathcal{G}$-reduced relative to $(\{v,v'\},V_2)$.  Write $y = y^{(1)} y^{(2)} y^{(3)}$ in an analogous way.

Since $x^{(1)}$ and $y^{(3)}$ are in $M_{V_2}$, we have
\[
E_{M_{V_2}}[x^{(1)}x^{(2)} x^{(3)} (u_0u_1)^k y^{(1)} y^{(2)} y^{(3)}] = x^{(1)} E_{M_{V_2}}[x^{(2)} x^{(3)} (u_0u_1)^k y^{(1)} y^{(2)}] y^{(3)}.
\]
Next, by Lemma \ref{lem: Pimnser Popa computation}, since $w^{(2)}$ is $(V_2,\{v,v'\})$-reduced and $\hat{w}^{(2)}$ is $(\{v,v'\},V_2)$-reduced, this equals
\[
x^{(1)} E_{M_{V_2}}[x^{(2)} x^{(3)} (u_0u_1)^k y^{(1)} y^{(2)}] y^{(3)} =  \varphi(x^{(2)} y^{(2)}) x^{(1)} E_{M_U}[x^{(3)} (u_0u_1)^k y^{(1)}] y^{(3)},
\]
where $U$ is the set of vertices in $V_2 \cap \{v,v'\}$ that are adjacent to all the letters in $w^{(2)}$.  Hence, in order to prove \eqref{eq: diffuseness claim 1} and hence finish (\ref{item: not adj vertices give diffuseness}) $\implies$ (\ref{item: diffuse rel char}), it suffices to show that
\[
\lim_{k \to \infty} \norm{E_{M_U}[x^{(3)} (u_0u_1)^k y^{(1)}]}_2 = 0.
\]
Since $v'$ is not in $V_2$, then $U$ must equal $\varnothing$ or $\{v\}$.  Hence, since $\varphi(x^{(3)} (u_0u_1)^k y^{(1)}) = \varphi \circ E_{M_v}[x^{(3)} (u_0u_1)^k y^{(1)}]$, it suffices to show that
\[
\lim_{k \to \infty} \norm{E_{M_v}[x^{(3)} (u_0u_1)^k y^{(1)}]}_2 = 0.
\]
However, for such $x^{(3)}$ and $y^{(1)}$ the above sequence is zero for sufficiently large $k$ by free independence (see also \cite[the proof of Proposition 3.16]{GKEPTConj})).\\

\noindent(\ref{item: diffuse rel implied by diffuse subalg}) $\implies$ (\ref{item: diffuse rel char})
Suppose that $M_{v}$ is diffuse for some $v \in V_1 \setminus V_2$, and thus exists a Haar unitary $u \in M_{v}$ (i.e., a unitary so that $\tau(u^k) = 0$ for any $k \in \bZ \setminus \set{0}$).  We claim that for $x, y \in M$, we have
\begin{equation} \label{eq: diffuseness claim 2}
\lim_{k \to \infty} \norm{E_{M_{V_2}}[xu^ky]}_2 = 0.
\end{equation}
As in the previous case, it suffices to consider $x$ and $y$ which are products of centered elements according to words $w$ and $\hat{w}$ respectively.  And again, we take a decomposition $w \approx w^{(1)} \cdot w^{(2)} \cdot w^{(3)}$ as in Lemma \ref{lem: word relative decomposition} with respect to $(V_2,\{v\})$ and a decomposition $\hat{w} \approx \hat{w}^{(1)} \cdot \hat{w}^{(2)} \cdot \hat{w}^{(3)}$ with respect to $(\{v\},V_2)$.  Let $x = x^{(1)} x^{(2)} x^{(3)}$ and $y = y^{(1)} y^{(2)} y^{(3)}$ be the resulting factorizations of $x$ and $y$.  Then
\[
E_{M_{V_2}}[x^{(1)}x^{(2)} x^{(3)} u^k y^{(1)} y^{(2)} y^{(3)}] = x^{(1)} E_{M_{V_2}}[x^{(2)} x^{(3)} u^k y^{(1)} y^{(2)}] y^{(3)} = x^{(1)} \varphi(x^{(2)} y^{(2)}) \varphi(x^{(3)} u^k y^{(1)}) y^{(3)},
\]
where the second equality follows from Lemma \ref{lem: Pimnser Popa computation}. Here the set $U$ is empty since $\{v\} \cap V_2 = \varnothing$.  Because $u$ is a Haar unitary, we have $u^{k}\to 0$ weakly as $k\to\infty$ and thus $\varphi(x^{(3)} u^k y^{(1)}) \to 0$. This completes the proof of \eqref{eq: diffuseness claim 2} and hence the proposition.
\end{proof}

\section{Bimodules from subgraphs and their fusion rules}\label{sec: fusion hah!}

Let $U \subseteq  W$.  We want to understand the basic construction of $M_U$ inside $M_W$. Hence, we want to understand $L^2(M_U,\varphi_U)$ as an $M_W$-$M_W$ bimodule.  More generally, for $V_1, V_2 \subseteq {W}$, we want to understand $M_W$ as an $M_{V_1}$-$M_{V_2}$-bimodule. We first recall a few facts about standard forms and Connes fusion of bimodules. 

Given a statial von Neumann algebra $(M,\varphi)$, recall that $L^2(M,\varphi)$ is an $M$-$M$-bimodule with actions
    \[
        x\cdot \xi \cdot y= x(J_\varphi y^* J_\varphi) \xi,
    \]
where $J_\varphi$ is the modular conjugation operator. We let $M\ni x\mapsto \hat{x}\in L^2(M,\varphi)$ denote the embedding determined by $\langle \hat{x},\hat{y}\rangle_\varphi = \varphi(y^*x)$. We will say $x \in M$ is \textbf{$\varphi$-analytic} if the modular automorphism group $\bR\ni t\mapsto \sigma_{t}^{\varphi}(x)$ has an extension to an entire function (such elements are dense by \cite[Lemma VIII.2.3]{TakesakiII}). In this case, for $z\in \bC$ we write $\sigma_z(x)$ for the image of $z$ under this (necessarily unique) entire extension. It follows that $\hat{y}\cdot x = ( y \sigma_{-i/2}(x))^{\widehat{}}$ whenever $x$ is $\varphi$-analytic and $y\in M$.

We will also need to consider the Connes fusion of bimodules over $\sigma$-finite von Neumann algebras. We refer the reader to \cite[Section 2]{OOTHaag} for general details, but for our purposes it suffices to consider the following special case. Let $(M,\varphi)$ and $(N,\psi)$ be statial von Neumann algebras, and let $B\subset M$ be a von Neumann subalgebra admitting a $\varphi$-preserving conditional expectation $E_B\colon M\to B$. If $\cH$ is a $B$-$N$-bimodule, then the $M$-$N$-bimodule
    \[
        L^2(M,\varphi)\underset{B}{\otimes} \cH
    \]
is formed by separation and completion of the algebraic tensor product $\widehat{M}\odot \cH$ with respect to
    \[
        \langle \hat{x}\otimes \xi, \hat{y}\otimes \eta\rangle := \langle E_B(y^*x)\cdot \xi, \eta\rangle.
    \]
We will denote the equivalence class of $\hat{x}\otimes \xi$ by $\hat{x}\otimes_B \xi$. We also note that 
    \[
        L^2(M,\varphi)\otimes_B L^2(B,\varphi|_B) \cong L^2(B,\varphi|_B)\otimes_B L^2(M,\varphi)\cong L^2(M,\varphi).
    \]
That is, $L^2(B,\varphi|_B)$ is an identity element with respect to the operation $\otimes_B$.

Let us now return to the context of graph products over $\cG=(\cV,\cE)$. For $V_1,V_2\subset \cV$, we will build a basis over $M_{V_{1}}$-$M_{V_{2}}$ by using orthonormal bases for $L^{2}(M_{v},\varphi_v)\ominus \bC\hat{1}$. Since we are not assuming that our von Neumann algebras have separable predual, we will not a priori be able to build an orthonormal basis for $L^{2}(M_{v},\varphi_v)\ominus \bC \hat{1}$ using elements of $M_{v}.$ For this reason, we will need to extend some of the results of Section~\ref{sec: cond expec formulas} to vectors in $L^{2}(M_{v},\varphi_v)\ominus \bC \hat{1}$. 

\begin{lem}\label{lem: extension of words to vectors}
Let $\cG=(\cV,\cE)$ be a graph, let $\{(M_v,\varphi_v)\colon v\in \cV\}$ be a family of statial von Neumann algebras, and let $(M,\varphi) = \gp_{v \in \cG} (M_v,\varphi_v)$.
\begin{enumerate}[(i)]
    \item Let $w=w_{1}\cdots w_{\ell}$ be a $\cG$-reduced word. Then there is a unique continuous multilinear map
        \[
            m\colon \prod_{i=1}^{\ell}(L^{2}(M_{v},\varphi_v)\ominus \bC\hat{1})\to L^{2}(M,\varphi)\ominus \bC\hat{1},
        \]
   such that $m(x_{1},\cdots,x_{\ell})=(x_{1}\cdots x_{\ell})^{\widehat{}}$ when $x_{i}\in M_{w_{i}}\cap \ker(\varphi_{w_{i}})$. Moreover,
    \[
        \|m(\xi_{1},\cdots,\xi_{\ell})\|_{\varphi}=\prod_{i=1}^{\ell}\|\xi_{j}\|_{\varphi},\qquad\qquad  \xi=(\xi_{1},\cdots,\xi_{\ell})\in \prod_{j=1}^{\ell}(L^{2}(M_{w_{j}},\varphi_{w_j})\ominus \bC \hat{1}).
    \]
   We denote $m(\xi_{1},\cdots,\xi_{\ell})=\xi_{1}\cdots\xi_{\ell}.$ \label{item:multilinear extension}
   \item  Let $w = w_1 \dots w_m$ and $\tilde{w} = \tilde{w}_1 \dots \tilde{w}_n$ be $\mathcal{G}$-reduced words. Set $\xi=\xi_{1}\cdots\xi_{m}$ and $\tilde{\xi}=\tilde{\xi}_{1}\cdots \tilde{\xi}_{n}$, where $\xi_j \in L^{2}(M_{w_j},\varphi_{w_{j}}) \ominus \bC\hat{1}$, $j=1,\ldots, m$, and $\tilde{\xi}_j \in  \ker L^{2}(M_{\tilde{w_j}},\varphi_{\tilde{w_{j}}}) \ominus \bC\hat{1}$, $j=1,\ldots,n$. If $w,\tilde{w}$ are not equivalent, then $\xi,\tilde{\xi}$ are orthogonal. If $w$ and $\tilde{w}$ are equivalent, then
        \[
            \langle \xi,\tilde{\xi}\rangle_\varphi=\prod_{j=1}^{m} \langle \xi_{j},\tilde{\xi}_{\sigma(j)}\rangle_\varphi
        \]
    where the permutation $\sigma: [m] \to [m]$ is the $\mathcal{G}$-monotone matching from $w$ to $\tilde{w}$ guaranteed by Lemma~\ref{lem:monotone mathcing is a bijection in reduced case}.
\label{item:inner products product formual equiv words}
\end{enumerate}

\end{lem}

\begin{proof}
(\ref{item:multilinear extension}): The uniqueness of $m$ follows from the density of $M_{w}\cap \ker(\varphi_{w})$ in $L^{2}(M_{w},\varphi_{w})\ominus \bC \hat{1}$. By Lemma~\ref{lem: reduced words orthogonal}, as well as density of $M_{w}\cap \ker(\varphi_{w})$ in $L^{2}(M_{w},\varphi_{w})\ominus \bC \hat{1}$, it follows that there is a unique isometry
    \[
        V\colon \bigotimes_{j=1}^{\ell}(L^{2}(M_{w_{j}},\varphi_{w_{j}})\ominus \bC\hat{1})\to L^{2}(M,\varphi)\ominus \bC\hat{1} 
    \]
such that $V(\hat{x}_{1}\otimes \cdots \otimes \hat{x}_{\ell})=(x_{1}\cdots x_{\ell})^{\widehat{}}.$
Setting $m(\xi_{1},\cdots,\xi_{\ell})=V(\xi_{1}\otimes \cdots \xi_{\ell})$ completes the proof.\\

\noindent(\ref{item:inner products product formual equiv words}): Observe that if
    \[
        \xi\in m\left(\prod_{j=1}^m (M_{w_j}\cap \ker(\varphi_{w_j}))\right) \qquad \text{ and }\qquad \tilde{\xi} \in m\left( \prod_{j=1}^n (M_{\tilde{w}_j}\cap \ker(\varphi_{\tilde{w}_j}))\right)
    \]
then the claim follows from Lemma \ref{lem: reduced words orthogonal}. The norm equality in (\ref{item:multilinear extension}) implies these sets are dense in 
    \[
        m\left(\prod_{j=1}^{n}(L^{2}(M_{w_{j}},\varphi_{w_{j}})\ominus \bC\hat{1})\right) \qquad \text{ and } \qquad  m\left(\prod_{i=1}^{n}(L^{2}(M_{\tilde{w}_{i}},\varphi_{\tilde{w}_{i}})\ominus \bC\hat{1})\right),
    \]
respectively, which completes the proof.
\end{proof}

\begin{remark}
Using Haagerup's theory of noncommutative $L^p$-spaces (see \cite{Haa79}), one can also make sense of $m(\xi_1,\ldots, \xi_n)= \xi_1\cdots \xi_n$ as a product of operators affiliated with the continuous core of $M$. The fact that such a produce remains in $L^2(M,\varphi)$ is a consequence of their relations via $\varphi$, which is determined by graph product structure of $M$.
\end{remark}

We will first analyze cyclic submodules generated by products over relatively $\cG$-reduced words.

\begin{lem}\label{lem: bimodule basis lemma}
Let $\cG=(\cV,\cE)$ be a graph, let $\{(M_v,\varphi_v)\colon v\in \cV\}$ be a family of statial von Neumann algebras, let $(M,\varphi) = \gp_{v \in \cG} (M_v,\varphi_v)$, and let $V_1, V_2 \subseteq \cV$.
For $w=w_1\ldots w_n$ a $\cG$-reduced word relative to $(V_1,V_2)$, let
    \[
        \xi=\xi_1\cdots \xi_n
    \]
where $\xi_j \in L^{2}(M_{w_j},\varphi_{w_{j}}) \ominus \bC\hat{1}$ with $\|\xi_j\|_\varphi=1$ for $j=1,\ldots, n$.
(In the case that $w$ is empty, we take by convention $\xi=\hat{1}$.) Let $H_\xi$ be the $M_{V_1}$-$M_{V_2}$-subbimodule of $L^2(M,\varphi)$ generated by $\xi$ and denote
    \[
        U := \{ v \in V_1 \cap V_2: v \sim w_j\ j=1,\ldots, \ell\}.
    \]
\begin{enumerate}[(i)]
\item There is a unique $M_{V_{1}}$-$M_{V_{2}}$ bimodular unitary $H_{\xi}\to L^2(M_{V_1},\varphi_{V_1}) \otimes_{M_U} L^2(M_{V_2},\varphi_{V_2})$ which sends $\xi$ to $1\otimes_{M_{U}}1$. \label{item: iso of bimodule to fusion}

\item If $\tilde{w}=\tilde{w}_1\cdots \tilde{w}_m$ is another $\cG$-reduced word relative to $(V_1,V_2)$ and $\tilde{\xi}=\tilde{\xi}_1\cdots \tilde{\xi}_m$ is a corresponding vector, then $H_\xi\perp H_{\tilde{\xi}}$ unless $w$ and $\tilde{w}$ are equivalent and $\langle \xi,\tilde{\xi}\rangle_\varphi\neq 0$.
\label{item: different bimod fusion ortho}
\end{enumerate}
\end{lem}
\begin{proof}
(\ref{item: iso of bimodule to fusion})
It suffices to show that 
\begin{equation}\label{eqn: GNS inner product equality}
    \ip{a\cdot (1\otimes_{M_{U}}1)\cdot b, 1\otimes_{M_{U}}1}=\ip{a\cdot \xi_{1}\cdots \xi_{n} \cdot b,\xi_{1}\cdots\xi_{n}}_\varphi,
\end{equation}
for all $a\in M_{V_{1}}$ and all $\varphi$-analytic $b\in M_{V_{2}}$. By Lemma \ref{lem: extension of words to vectors} for fixed $a,b$ the right-hand side is a continuous function of $(\xi_{1},\cdots,\xi_{n})\in \prod_{j=1}^{n}(L^{2}(M_{w_{j}},\varphi_{w_{j}})\ominus \bC \hat{1})$. Thus, by density of $M_{w}\cap \ker(\varphi_{w})$ in $L^{2}(M_{w},\varphi_w)\ominus \bC\hat{1}$, we may reduce to the case where $\xi_{j}=x_{j}$ where $x_{j}\in M_{w_{j}}\cap \ker(\varphi_{w_{j}})$ and $\varphi(x_{j}^{*}x_{j})=1$. In this case, set $x=x_{1}\cdots x_{j}$ so that $\xi=\hat{x}$.

The left-hand side of (\ref{eqn: GNS inner product equality})  is:
\[\ip{\widehat{a}\otimes_{M_{U}}(\sigma_{-i/2}(b))^{\widehat{}}, 1\otimes_{M_{U}}1}=\varphi(E_{M_{U}}(a)\sigma_{-i/2}(b)),\]
and the right-hand side of (\ref{eqn: GNS inner product equality}) is: 
\[\ip{(ax\sigma_{-i/2}(b))^{\widehat{}},\widehat{x}}_\varphi=\varphi(x^{*}ax\sigma_{-i/2}(b))=\varphi(E_{M_{V_{2}}}(x^{*}ax)\sigma_{-i/2}(b)).\]
Thus Lemma \ref{lem: Pimnser Popa computation} implies (\ref{eqn: GNS inner product equality}).\\

\noindent(\ref{item: different bimod fusion ortho}): 
It is enough to show that for all  $a\in M_{V_{1}}$ and $\varphi$-analytic $b\in M_{V_{2}}$ that 
\begin{equation}\label{eqn: GNS inner product equality 2}
   \langle a\cdot \xi \cdot b,\tilde{\xi}\rangle_\varphi=0,
\end{equation}
if $w,\tilde{w}$ are not equivalent, and that 
    \[
        \langle a\cdot \xi \cdot b,\tilde{\xi}\rangle_\varphi=\langle \xi,\tilde{\xi} \rangle_\varphi \varphi(E_{M_{U}}(a)\sigma_{-i/2}(b))
    \]
if $w,\tilde{w}$ are equivalent. 
As in (\ref{item: iso of bimodule to fusion}) we may reduce to the case that $\xi=\widehat{x},$ $\tilde{\xi}=\widehat{\tilde{x}}$, where $x=x_{1}\cdots, x_{n}$, $\tilde{x}=\tilde{x}_{1}\cdots \tilde{x}_{m}$, and $x_{j}\in M_{w_{j}}\cap \ker(\varphi_{w_{j}})$ and $\tilde{x}_{i}\in M_{\tilde{w}_{i}}\cap \ker(\varphi_{\tilde{w}_{i}})$. 
We then have
    \[
        \langle a\cdot \widehat{x} \cdot b,\widehat{\tilde{x}} \rangle_\varphi =\varphi((\tilde{x})^{*}ax\sigma_{-i/2}(b))=\varphi(E_{M_{V_{2}}}((\tilde{x})^{*}ax)\sigma_{-i/2}(b)),
    \]
so that our desired conclusion follows from Lemma \ref{lem: Pimnser Popa computation}.
\end{proof}

Our main result in this section provides a classification of $L^{2}(M)$ as a bimodule over two subalgebras coming from induced subgraphs. This also yields the first part of Theorem~\ref{thm: bimodules intro}.

\begin{thm} \label{thm: bimodules}
Let $\cG=(\cV,\cE)$ be a graph, let $\{(M_v,\varphi_v)\colon v\in \cV\}$ be a family of statial von Neumann algebras, let $(M,\varphi) = \gp_{v \in \cG} (M_v,\varphi_v)$, and let $V_1, V_2 \subseteq \cV$.
For each $U\subset V_1\cap V_2$, denote by $\cW_{\cG}(V_1,V_2,U)$ the set of $\cG$-reduced words relative to $(V_1,V_2)$ of the form $w_1\cdots w_{\ell}$ satisfying $U=\{v\in V_1\cap V_2\colon v\sim w_j\ j=1,\cdots,\ell\}$.
Set
    \[
        k_{\cG}(V_1,V_2,U) := \sum_{w_1\cdots w_\ell \in \cW_{\cG}(V_1, V_2, U)} \prod_{j=1}^\ell (\dim(L^{2}(M_{w_j},\varphi_{w_{j}}))- 1).
    \]
Then one has
    \begin{equation} \label{eq: bimodule multiplicity}
        _{M_{V_1}} L^2(M,\varphi)_{M_{V_2}} \cong \bigoplus_{U \subseteq V_1 \cap V_2} (_{M_{V_1}} L^2(M_{V_1},\varphi_{V_1}) \underset{M_U}{\otimes} L^2(M_{V_2},\varphi_{V_2})_{M_{V_2}})^{\oplus k_{\cG}(V_1,V_2,U)}.
    \end{equation}
\end{thm}
\begin{proof}
For each $v \in \mathcal{V}$, fix an orthonormal basis $\mathcal{B}_v$ for $L^2(M_v,\varphi_{v}) \ominus \bC$.  By Lemma~\ref{lem: bimodule basis lemma}, the $M_{V_1}$-$M_{V_2}$-bimodules
    \[
        \{H_\xi: \xi = \xi_1 \dots \xi_\ell,\ w = w_1 \dots w_\ell \in \mathcal{W}_{\cG}(V_1,V_2,U),\ \xi_j \in \mathcal{B}_{w_j} \text{ for } j = 1, \dots, \ell \}
 \]
are mutually orthogonal and satisfy $H_\xi \cong  L^2(M_{V_1},\varphi_{V_1}) \otimes_{M_U} L^2(M_{V_2},\varphi_{V_2})$ where $U = \{ v \in V_1 \cap V_2: v \sim w_j\ j=1,\ldots, \ell\}$.  For each $U \subseteq V_1 \cap V_2$, the number of copies of $_{M_{V_1}} L^2(M_{V_1},\varphi_{V_1}) \otimes_{M_U} L^2(M_{V_2},\varphi_{V_2})_{M_{V_2}}$ is given by \eqref{eq: bimodule multiplicity}, since $\dim(L^{2}(M_{v},\varphi_{v})\ominus \bC\hat{1})=\dim(L^{2}(M_{v},\varphi_{v}))-1$.

The proof of the direct sum decomposition will be complete once we verify that the bimodules $H_\xi$ span a dense subset of $L^2(M,\varphi)$. From Lemma~\ref{lem: reduced words span}, we know that $L^2(M,\varphi)$ is densely spanned by $\xi_1 \dots \xi_\ell$ for $\xi_j \in \mathcal{B}_{w_j}$ for reduced words $w_1 \dots w_\ell$. 
By Lemma~\ref{lem: word relative decomposition}, an
arbitrary reduced word $w$ is equivalent to a word of the form $v\cdot w'\cdot u$ where $v$ is a reduced word in $V_1$, $u$ is a reduced word in $V_2$, and $w'$ is reduced relative to $(V_{1},V_{2})$. This shows that the span of the subspaces $\cH_{\xi}$ contain all $\xi_1 \dots \xi_\ell$ for $x_j \in \mathcal{B}_{w_j}$ for reduced words $w_1 \dots w_\ell$, and thus the bimodules $\cH_{\xi}$ densely span $L^{2}(M,\varphi)$.
\end{proof}

For future applications to relative amenability, we determine the fusion rules for these bimodules in the following proposition. This also gives the rest of Theorem~\ref{thm: bimodules intro}.

\begin{prop} \label{prop: fusion}
For $V_1, V_2 \subseteq \mathcal{V}$ and $U \subseteq V_1 \cap V_2$, denote
    \[
        \mathscr{H}_U(V_1,V_2) := _{M_{V_1}} L^2(M_{V_1},\varphi_{V_1}) \underset{M_{U}}{\otimes} L^2(M_{V_2},\varphi_{V_2})_{M_{V_2}},
    \]
and denote by $\cG_2$ the subgraph of $\cG$ induced by $V_2$. Then we have the following fusion rules: for $U_1 \subseteq V_1 \cap V_2$ and $U_2 \subseteq V_2 \cap V_3$,
\[
\mathscr{H}_{U_1}(V_1,V_2) \otimes_{M_{V_2}} \mathscr{H}_{U_2}(V_2,V_3)
\cong
\bigoplus_{W \subseteq U_1 \cap U_2} \mathscr{H}_W(V_1,V_3)^{\oplus k_{\cG_2}(U_1,U_2,W)}.
\]
\end{prop}

\begin{proof}
First, observe that
    \begin{align*}
        \mathscr{H}_{U_1}(V_1,V_2) & \underset{M_{V_2}}{\otimes} \mathscr{H}_{U_2}(V_2,V_3) \\
            &= 
        (_{M_{V_1}} L^2(M_{V_1},\varphi_{V_1}) \underset{M_{U_1}}{\otimes} L^2(M_{V_2},\varphi_{V_2}))_{M_{V_2}})
            \underset{M_{V_2}}{\otimes}
        (_{M_{V_2}} L^2(M_{V_2},\varphi_{V_2})) \underset{M_{U_2}}{\otimes} L^2(M_{V_3},\varphi_{V_3}))_{M_{V_3}}) \\
            &\cong
        (_{M_{V_1}} L^2(M_{V_1},\varphi_{V_{1}}) \underset{M_{U_1}}{\otimes} L^2(M_{V_2},\varphi_{V_2})) \underset{M_{U_2}}{\otimes} L^2(M_{V_3},\varphi_{V_3})_{M_{V_3}}).
    \end{align*}
Then
    \[
        _{M_{U_1}} L^2(M_{V_2},\varphi_{V_2})_{M_{U_2}} \cong \bigoplus_{W \subseteq U_1 \cap U_2} (_{M_{U_1}} L^2(M_{U_1},\varphi_{U_1}) \underset{M_W}{\otimes} L^2(M_{U_2},\varphi_{U_2})_{M_{U_2}})^{\oplus k_{\cG_2}(U_1,U_2,W)}
    \]
Applying  $L^2(M_{V_1}) \otimes_{M_{U_1}}$ on the left and $\otimes_{M_{U_2}} L^2(M_{V_2})$ on the right, we get
\begin{align*}
\mathscr{H}_{U_1}(V_1,V_2) \otimes_{M_{V_2}} \mathscr{H}_{U_2}(V_2,V_3) &\cong \bigoplus_{W \subseteq U_1 \cap U_2} (_{M_{V_1}} L^2(M_{V_1},\varphi) \underset{M_W}{\otimes} L^2(M_{V_2},\varphi)_{M_{V_2}})^{\oplus k_{\cG_2}(U_1,U_2,W)} \\
&= \bigoplus_{W \subseteq U_1 \cap U_2} \mathscr{H}_W(V_1,V_2)^{\oplus k_{\cG_2}(U_1,U_2,W)}.\qedhere
\end{align*}
\end{proof}

As a sample application, we characterize weak coarseness of subalgebras corresponding to an induced subgraphs.  This characterization of coarseness is stated in terms of amenability of certain subalgebras, which we provide a complete characterization of in Proposition~\ref{prop: amenability characterization}.

\begin{theorem}\label{thm: weakly coarse}
Let $\mathcal{G}=(\mathcal{V},\mathcal{E})$ be a graph,  let $\{(M_v,\varphi_v):v \in \cV\}$ be a family of stacial von Neumann algebras, and let  $(M,\varphi)=\gp_{v\in \mathcal{V}} (M_v,\varphi_v)$. Assume $\dim(M_v)\geq 2$ for all $v\in V$. For $V_0 \subseteq \mathcal{V}$, $L^2(M,\varphi) \ominus L^2(M_{V_0},\varphi_{V_0})$ is weakly coarse as an $M_{V_0}$-$M_{V_0}$-bimodule if and only if $M_{\gstar{v}\cap V_0}$ is amenable for all $v\not\in V_0$.
\end{theorem}
\begin{proof}
First, suppose that $M_{\gstar{v}\cap V_0}$ is amenable for all $v \not \in V_0$.  By Theorem~\ref{thm: bimodules}, $_{M_{V_0}} L^2(M,\varphi)_{M_{V_0}}$ is a direct sum of bimodules of the form $L^2(M_{V_0},\varphi_{V_0}) \otimes_{M_U} L^2(M_{V_0},\varphi_{V_0})$ where $U = \{v \in V_0: v \sim w_j\ j=1,\ldots, \ell\}$ for some word $w_1\cdots w_\ell$ that is $\mathcal{G}$-reduced relative to $(V_0,V_0)$.  To obtain the orthogonal complement of $L^2(M_{V_0},\varphi_{V_0})$, one sums over the \emph{nonempty} words of this form with the appropriate multiplicity.  Note that $U \subseteq \gstar{w_1} \cap V_0$, and $w_1\not\in V_0$ since $w_1\cdots w_e\ll$ is nonempty and $\cG$-reduced relative to $(V_0,V_0)$.  Thus $M_{\gstar{w_1} \cap V_0}$ is amenable by assumption, and since there is a faithful normal conditional expectation from this algebra on $M_U$, we also have that $M_U$ is amenable.  Thus, 
    \[
        _{M_U} L^2(M_U,\varphi_U)_{M_U} \prec {}_{M_U} L^2(M_U,\varphi_U) \otimes L^2(M_U,\varphi_U)_{M_U}.
    \]
by \cite[Corollary A.2]{BMOCoAmenable} (see also \cite{Connes} for the separable predual case).
Now we apply $_{M_{V_0}} L^2(M_{V_0},\varphi_{V_0}) \otimes_{M_U}$ on the left and apply $\otimes_{M_U} L^2(M_{V_0},\varphi_{V_0})_{M_{V_0}}$ on the right to obtain
    \[
        _{M_{V_0}} L^2(M_{V_0},\varphi_{V_0}) \otimes_{M_U} L^2(M_{V_0},\varphi_{V_0})_{M_{V_0}} \prec {}_{M_{V_0}} L^2(M_{V_0},\varphi_U) \otimes L^2(M_{V_0},\varphi_{V_0})_{M_{V_0}}.
    \]
where we have used the fact that weak containment is preserved under Connes fusion \cite[Proposition 2.2.1]{PopaCorr}.  Taking the direct sum over all such nonempty words $w_1\cdots w_\ell$ yields that $L^2(M,\varphi) \ominus L^2(M_{V_0}\varphi_{V_0})$ is weakly coarse over $M_{V_0}$.

Conversely, suppose there exists some vertex $v \not \in V_0$ such that $M_{\gstar{v} \cap V_0}$ is non-amenable. For ease of notation, denote $V_1 := \gstar{v} \cap V_0$.  Fix $x \in M_v$ with $\varphi_v(x) = 0$ and $\varphi_v(x^*x)=1$.  Let $H_{\hat{x}}$ be the $M_{V_0}$-$M_{V_0}$-subbimodule of $L^2(M,\varphi)$ generated by $\hat{x}$, which we note is in $L^2(M,\varphi) \ominus L^2(M_{V_0},\varphi_{V_0})$ since $v\not\in V_0$.  Applying Lemma~\ref{lem: bimodule basis lemma}.(\ref{item: iso of bimodule to fusion}) to $V_2:=V_1$ and $w=v$ (so that $U=V_1$), we have that $M_{V_1}$-$M_{V_1}$-bimodule generated by $\hat{x}$ is isomorphic to $L^2(M_{V_1},\varphi_{V_1})\otimes_{M_{V_1}}L^2(M_{V_1},\varphi_{V_1}) \cong L^2(M_{V_1},\varphi_{V_1})$. In particular, since $M_{V_1}$ is not amenable, this $M_{V_1}$-$M_{V_1}$-bimodule is not weakly coarse.  Since it is an $M_{V_1}$-$M_{V_1}$-subbimodule of $H_{\hat{x}}$, it follows that $_{M_{V_1}} (H_{\hat{x}})_{M_{V_1}}$ is not weakly coarse, and in turn $_{M_{V_0}} (H_{\hat{x}})_{M_{V_0}}$ is not weakly coarse.
\end{proof}

\begin{remark}
The previous theorem also can recover, with a different approach, the characterization of fullness of graph products of amenable von Neumann algebras in Theorem~\ref{thm: char diffuse factor full intro}. \end{remark}

\section{Relative amenability via bimodules}\label{sec: rel amen}

A useful implication of relative amenability is the following.  By \cite[Section 2.2]{BMOCoAmenable}, we  can described the standard form of $\ip{M,e_{B}}$ via the isomorphism
    \[
        L^{2}(\ip{M,e_{B}})\cong L^{2}(M)\otimes_{B}L^{2}(M),
    \]
as $M$-$M$ bimodules. Consequently, \cite[Corollary A.2]{BMOCoAmenable} tells us that $A$ being amenable relative to $B$ inside $M$ implies that $L^{2}(M)$ is weakly contained in $L^{2}(\ip{M,e_{B}})$ as $A$-$A$ bimodules. Note that---due to the conditional expectation being required to be normal on $M$---the converse is not a priori true.
However, in our setting $M$ will be a graph product and $A,B$ will be subalgebras corresponding to induced subgraphs. In this case, the detailed analysis of the previous section will lead us to a complete classification of when $L^{2}(M)$  is weakly contained in $L^{2}(\ip{M,e_{B}})$. From this classification, we will be able to directly argue that if $L^{2}(M)$  is not weakly contained in $L^{2}(\ip{M,e_{B}})$, then $A$ must be amenable relative to $B$ inside of $M$.

As the fusion rules provided in Proposition~\ref{prop: fusion} decompose relative tensor products as direct sums, we highlight the fact that, in the factorial case, bimodules weakly contained in direct sums are necessarily weakly contained in one of the summands. Indeed, suppose that $M$ is factor, and for a faithful normal state $\varphi$ on $M$ let $J_\varphi$ be the associated modular conjugation on $L^2(M,\varphi)$. Then the induced map $\pi\colon M\otimes_{\max{}}M^{op}\to B(L^{2}(M,\psi))$ satisfying $\pi(a\otimes b^{op})=aJ_\varphi b^{*}J_\varphi$ has trivial commutant ($\pi(M\otimes_{\max{}} M^{op})' = M'\cap (J_\varphi M J_\varphi)'= M\cap M'$) and is thus irreducible \cite[Proposition I.9.20]{TakesakiI}. Hence the state on $M\otimes_{\max{}}M^{op}$ given by $x\mapsto \ang{\pi(x)\hat{1},\hat{1}}$ is an extreme point of the state space \cite[Theorem I.9.22]{TakesakiI}, and so a minor modification of the proof of \cite[Theorem 1.5]{FellSpaces}  gives the following.

\begin{lem}\label{lem: inducing to a direct summand}
Let $M$ be a factor and $\varphi$ a faithful normal state on $M$. Suppose $\cH_1,\ldots, \cH_n$ are $M$-$M$ bimodules with 
    \[
        L^{2}(M,\varphi)\prec \cH_1\oplus \cdots \oplus \cH_n
    \]
 as $M$-$M$ bimodules. Then there is an $1\leq i\leq n$ so that $L^{2}(M,\varphi)\prec \cH_i$ as $M$-$M$ bimodules.
\end{lem}

It will also be helpful to prove the following general lemma, which will ultimately reduce our work of checking when one subalgebra corresponding to an induced subgraph is amenable relative to another, to the case  of smaller subgraphs. 
\begin{lem}\label{lem: relative semidisrcete for subgraph}
Let $\cG=(\cV,\cE)$ be a graph, let $\{(M_v,\varphi_v)\colon v\in \cV\}$ be a family of statial von Neumann algebras, let $(M,\varphi) = \gp_{v \in \cG} (M_v,\varphi_v)$, and let $V_1, V_2 \subseteq \cV$. Suppose that $_{M_{V_{1}}} L^2(M,\varphi)_{M_{V_1}}$ is weakly contained in $_{M_{V_{1}}} L^2(M,\varphi) \otimes_{M_{V_2}} L^2(M,\varphi)_{M_{V_1}}$.
    \begin{enumerate}[(i)]
    \item $M_{V_0}$ is amenable for all $V_0\subseteq V_1\setminus V_2$. \label{item: rel amen vacuous case}
    \item If $M_{V_0}$ is factor for $V_0\subset V_1$, then there exists $U\subseteq V_0\cap V_2$ (possibly empty) so that $M_{V_0}$ is amenable relative to $M_{U}$.\label{item: rel amen factor case}
\end{enumerate}
\end{lem}
\begin{proof}
We first make a preliminary observation. For $V_0\subset \cV$ and $A\subset V_0\cap V_2$, we adopt the notation from Proposition~\ref{prop: fusion} and denote
    \[
        \mathscr{H}_A(V_0,V_2):= _{M_{V_0}} L^2(M_{V_0},\varphi_{V_0}) \underset{M_A}{\otimes} L^2(M_{V_2},\varphi_{V_2})_{M_{V_2}}.
    \]
By Theorem \ref{thm: bimodules}, we have
    \[
        _{M_{V_0}} L^2(M,\varphi)_{M_{V_2}} \cong \bigoplus_{A \subseteq V_0\cap V_2 } \mathscr{H}_A(V_0,V_2)^{\oplus k_{\cG}(V_0,V_2,A)} \subseteq \bigoplus_{A \subseteq V_0\cap V_2} \mathscr{H}_A(V_0,V_2)^{\oplus \infty}.
    \]
Therefore, using Proposition~\ref{prop: fusion} we have
    \[
        _{M_{V_0}} L^2(M,\varphi) \underset{M_{V_2}}{\otimes} L^2(M,\varphi)_{M_{V_0}} \subseteq \bigoplus_{A,B\subseteq V_0\cap V_2} \left( \mathscr{H}_A(V_0,V_2) \underset{M_{V_2}}{\otimes} \mathscr{H}_B(V_2, V_0) \right)^{\oplus \infty} \subseteq \bigoplus_{U\subseteq V_0\cap V_2} \mathscr{H}_U(V_0,V_0)^{\oplus \infty}.
    \]
By assumption, $_{M_{V_{1}}} L^2(M,\varphi)_{M_{V_1}}$ is weakly contained in $_{M_{V_{1}}} L^2(M,\varphi) \otimes_{M_{V_2}} L^2(M,\varphi)_{M_{V_1}}$. If $V_0\subset V_1$, then by restriction we have that $_{M_{V_0}} L^2(M,\varphi)_{M_{V_0}}$ is weakly contained in $_{M_{V_0}} L^2(M,\varphi) \otimes_{M_{V_2}} L^2(M,\varphi)_{M_{V_0}}$, and so the above shows that 
    \begin{align}\label{eqn:weak containment computation}
        _{M_{V_0}} L^2(M,\varphi)_{M_{V_0}} \prec \bigoplus_{U\subseteq V_0\cap V_2} \mathscr{H}_U(V_0,V_0)^{\oplus \infty}.
    \end{align}
    
Now, if $V_0\subset V_1\setminus V_0$, then the only term in the above direct sum  corresponds to $U=\varnothing$, which has $M_U=\bC$. Thus the above gives
    \[
         _{M_{V_0}} L^2(M,\varphi)_{M_{V_0}} \prec \ _{M_{V_0}} L^2(M_{V_0},\varphi_{V_0}) \otimes L^2(M_{V_0},\varphi_{V_0})_{M_{V_0}} \cong\ _{M_{V_0}} L^2(M_{V_0},\varphi_{V_0}) \otimes \overline{L^2(M_{V_0},\varphi_{V_0})}_{M_{V_0}}.
    \]
Note that the bimodule in the last expression is equivalent to the standard form of $B(L^2(M_{V_0},\varphi_{V_0}))$ with respect to its trace. Hence $M_{V_0}$ is amenable by \cite[Corollary A.2]{BMOCoAmenable} (see also \cite{Connes} in the case of separable preduals), which proves (\ref{item: rel amen vacuous case}).

To prove (\ref{item: rel amen factor case}), suppose $M_{V_0}$ is a factor for $V_0\subset V_1$. Since $\cG$ is a finite graph, the direct sum over $U\subset V_0\cap V_2$ in (\ref{eqn:weak containment computation}) only has finitely many terms,  and hence Lemma~\ref{lem: inducing to a direct summand} implies
    \[
        _{M_{V_0}} L^2(M,\varphi)_{M_{V_0}} \prec\ _{M_{V_0}} L^2(M_{V_0},\varphi_{V_0}) \underset{M_U}{\otimes} L^2(M_{V_0},\varphi_{V_0})_{M_{V_0}},
    \]
for some $U\subset V_0\cap V_2$. By \cite[Section 2.2]{BMOCoAmenable}, the latter bimodule is isomorphic to the standard form for $\ip{M_{V_0}, e_{M_U}}$. Thus \cite[Corollary A.2]{BMOCoAmenable} yields is a conditional expectation $\Phi\colon \ip{M_{V_{0}},e_{M_{U}}}\to M_{V_{0}}$ so that  $M_{V_0}$ is amenable relative to $M_{U}$.
\end{proof}

\subsection{Proofs of Theorems~\ref{thm: complete char rel amen tracial intro} and \ref{thm: complete char rel amen stacial intro}}

Let us first reduce Theorem~\ref{thm: complete char rel amen tracial intro} to Theorem~\ref{thm: complete char rel amen stacial intro}. 
Comparing the two theorems, this amounts to showing that if $(M_i,\tau_i)$ is a tracial von Neumann algebra admitting a trace zero unitary for $i=1,2$, then the following are equivalent:
    \begin{enumerate}[(I)]
        \item $\dim(M_1)=\dim(M_2)=2$;
        \item $M_1*M_2$ is amenable;
        \item $M_1*M_2$ is amenable relative to $M_1$.
    \end{enumerate}
The equivalence of the first two items is well known (see, for example, \cite[Theorem 2]{ChingFreeProd}), and (II) implies (III) follows from the definition. So now suppose (III) holds. Applying Proposition~\ref{prop: relative diffuseness} to the graph $\cG=(\{1,2\}, \varnothing)$ with $V_1=\{1,2\}$ and $V_2=\{1\}$, we see that $M_{V_1}=M_1*M_2$ is diffuse relative to $M_{V_2}=M_1$ inside $M_1*M_2$. That is, $M_1*M_2$ does not intertwine into $M_1$ inside of $M_1*M_2$, and thus \cite[Corollary 2.12]{CartanAFP} implies (II).\\

We now prove Theorem~\ref{thm: complete char rel amen stacial intro}. First assume that Theorem~\ref{thm: complete char rel amen stacial intro}.(\ref{item: amenable vertex wise intro}) and (\ref{item: rel amen free product intro})  hold.  Let $P_1$, \dots, $P_n$ be the pairs of vertices $\{v,w\}$ where $v \in V_1 \setminus V_2$, $w \in V_1$, and $v$ and $w$ are not adjacent.  Denote $Q_1 := V_1 \setminus (V_2 \cup P_1 \cup \dots \cup P_n)$ and $Q_2 := V_1 \cap V_2 \setminus (P_1 \cup \dots \cup P_n)$.  By (\ref{item: relative amen adjacency constraint intro}), all the vertices in each $P_j$ are connected to all other vertices in $Q_1 \cup Q_2$.  Moreover, each $v \in Q_1$ is connected to all vertices in $V_1$ by definition of $Q_1$.  Thus,
    \[
        M_{V_1} = \left( \overline{\bigotimes}_{j=1}^n M_{P_j} \right) \bar\otimes M_{Q_1} \bar{\otimes} M_{Q_2}.
    \]
and
    \[
        M_{V_1 \cap V_2} = \left( \overline{\bigotimes}_{j=1}^n M_{P_j \cap V_2} \right) \bar{\otimes} \,\bC\, \bar{\otimes} M_{Q_2}.
    \]
By assumption (\ref{item: case of two vertices relative amen intro}), $M_{P_j}$ is amenable relative $M_{P_j \cap V_2}$ in $M_{P_j}$ for each $j=1,\ldots, n$. By assumption (\ref{item: amenable vertex wise intro}) and \cite[Theorem 6]{Connes}, $M_{Q_1}$ is amenable. 
Thus Lemma~\ref{lem: tensor rel amen} implies that $M_{V_1}$ is amenable relative to $M_{V_1 \cap V_2}$ (inside $M_{V_1}$). By Lemma~\ref{lem: commutating squares rel amen reduction}, this in turn implies that $M_{V_{1}}$ is amenable relative to $M_{V_{2}}$ in $M$.

Conversely, suppose that $M_{V_1}$ is amenable relative to $M_{V_2}$ inside $M$. Recall from the discussion at the beginning of Section~\ref{sec: rel amen} that this implies $L^2(M,\varphi)$ is weakly contained in $L^2(M,\varphi) \otimes_{M_{V_2}} L^2(M,\varphi)$ as $M_1$-$M_1$-bimodules. Thus for each $v\in V_1\setminus V_2$ we can apply Lemma~\ref{lem: relative semidisrcete for subgraph} to $V_0=\{v\}$ to obtain that $M_v$ is amenable. This gives Theorem~\ref{thm: complete char rel amen stacial intro}.(\ref{item: amenable vertex wise intro}). To prove Theorem~\ref{thm: complete char rel amen stacial intro}.(\ref{item: rel amen free product intro}), let $v \in V_1 \setminus V_2$ and $w \in V_1$ with $w \neq v$ and assume $v$ and $w$ are not adjacent. We will show that (\ref{item: case of two vertices relative amen intro}) and (\ref{item: relative amen adjacency constraint intro}) must occur. 

For (\ref{item: case of two vertices relative amen intro}), first note that if $\dim(M_{v})=\dim(M_{w})=2$, then by \cite[Theorem 1.1]{DykemaFreeproductsHyper} we have that $M_{\{v,w\}}=M_{v}*M_{w}$ is amenable. In particular, $M_{\{v,w\}}$ is also amenable relative to $M_{w}$, proving (\ref{item: case of two vertices relative amen intro}) in this case. If $\max(\dim(M_{v},M_{w}))\geq 3$, then
\cite[Theorem 4.1 and Remark 4.2]{UedaTypeIIIfreeproduct} implies that $M_{\{v,w\}}=M_{v}*M_{w}$ is a factor, and thus
Lemma~\ref{lem: relative semidisrcete for subgraph} applied to $V_0=\{v,w\}$ yields that $M_{\{v,w\}}$ is amenable relative to $M_{U}$ for some $U\subseteq V_0\cap V_{2}$. Noting that $w\not\in V_2$ forces $U=\varnothing$, we see that in this case $M_{\{v,w\}}$ is amenable. If $w\in V_2$, then either $U=\{w\}$ or $U=\varnothing$, but in both cases one has that $M_{\{v,w\}}$ is amenable relative to $M_w$. We have thus established (\ref{item: case of two vertices relative amen intro}).

For (\ref{item: relative amen adjacency constraint intro}), consider another vertex $u \in V_1 \setminus \{v,w\}$, and suppose towards a contradiction that one of $v$ or $w$ is not adjacent to $u$. Note that this implies the subgraph $\cG_0$ induced by $V_0:=\{v,w,u\}$ is join-irreducible, and hence $M_{V_0}$ is a factor by Theorem~\ref{thm:factoriality}. Consequently, Lemma~\ref{lem: relative semidisrcete for subgraph} implies $M_{V_0}$ is amenable relative to $M_U$ for some $U\subseteq V_0\cap V_2$. Since $U \subset \{w,u\} \subset V_0$, it follows that $M_{V_0}$ is amenable relative to $M_{\{w,u\}}$. We will show this is a contradiction by way of Lemma~\ref{lem: free product non-amenability} using the observation that
    \[
        M_{V_0} \cong M_{\{v,u\}} *_{M_u} M_{\{w,u\}},
    \]
where the amalgamated free product is taken with respect to the $\varphi$-preserving conditional expectations. Let $u_0$ and $u_2$ be state zero unitaries in $M_v^{\varphi_v}$ and $M_w^{\varphi_w}$, respectively, so that $E_{M_u}[u_0]=\varphi_v[u_0]=0$ and similarly $E_{M_u}[u_2]=0$. Also let $x$ be a state zero unitary in $M_u^{\varphi_u}$. If $v$ is not adjacent to $u$, then $u_1:=xu_0 x^*$ satisfies 
    \[
        E_{M_u}[u_1] = x E_{M_u}[u_0] x^*=0,
    \]
and by free independence
    \[
        E_{M_u}[u_0^* u_1] = \varphi_v(u_0^*) x \varphi_v(u_0) x^*=0.
    \]
Consequently, Lemma~\ref{lem: free product non-amenability} gives the contradiction that $M_{V_0}$ is not amenable relative to $M_{\{w,u\}}$. If instead $w$ is not adjacent to $u$, then we instead define $u_1:=x u_2 x^*$ and argue as above to get $E_{M_u}[u_1]=E_{M_u}[u_2^* u_1]=0$, which once again gives a contradiction via Lemma~\ref{lem: free product non-amenability}. Thus we we must have that both $v$ and $w$ are adjacent to $u$, establishing (\ref{item: relative amen adjacency constraint intro}).

\subsection{Amenability by way of relative amenability}

In this section we characterize when a graph product $(M,\varphi) = \gp_{v \in \mathcal{G}} (M_v,\varphi_v)$ is amenable by specializing to the case where $V_1 = \cV$ and $V_2 = \varnothing$. The characterization can be read off from Theorem~\ref{thm: complete char rel amen stacial intro}, but in fact, we claim that this characterization holds even \emph{without} the assumption that $M_v^{\varphi_v}$ contains a state zero unitary.

\begin{prop} \label{prop: amenability characterization}
Let $\cG=(\cV,\cE)$ be a graph, let $\{(M_v,\varphi_v)\colon v\in \cV\}$ be a family of statial von Neumann algebras, and let $(M,\varphi) = \gp_{v \in \cG} (M_v,\varphi_v)$. Assume $\dim M_v \geq 2$.  Then $M$ is amenable if and only if the following conditions hold:
\begin{enumerate}[(1)]
    \item For each $v \in \cV$, $M_v$ is amenable.
    \item If $v$ and $w$ are not adjacent in $\mathcal{G}$, then $\dim(M_v) = \dim(M_w) = 2$ and $v$ and $w$ are adjacent to all the other vertices.
\end{enumerate}
\end{prop}

\begin{proof}
First, suppose that (1) and (2) hold.  Let $\mathcal{G} = \mathcal{G}_1 + \dots + \mathcal{G}_n$ be the graph join decomposition of $\mathcal{G}$.  We claim that each $\mathcal{G}_j$ is either a single vertex or a pair of non-adjacent vertices. Indeed, if $v$ is a vertex in $\cG_j$ that is adjacent to all other vertices in $\cG_j$, then it is adjacent to all vertices in $\cG$ and hence $\cG_j=(\{v\},\varnothing)$. Otherwise, there exists another vertex $w$ in $\cG_j$ that is not adjacent to $v$. But then (2) implies $\cG_j=(\{v,w\},\varnothing)$.  Writing $(N_j,\psi_j) = \gp_{v \in \mathcal{G}_j} (M_v,\varphi_v)$, we have
    \[
        (M,\varphi) \cong (N_1,\psi_1)\bar\otimes \cdots \bar\otimes (N_n,\psi_n).
    \]
If $\cG_j$ has one vertex, then $(N_j,\psi_j)$ is amenable by (1). If $\cG_j$ has two vertices, then $(N_j,\psi_j) = (M_v,\varphi_v) * (M_w, \varphi_w)$ where $\dim(M_v) = \dim(M_w) = 2$ by (2), and hence is amenable by \cite[Theorem 2]{ChingFreeProd}.  Thus, $M$ is amenable as a tensor product of amenable von Neumann algebras.

Conversely, suppose that $M$ is amenable.  Recall that for any $U \subseteq \cV$, there is a faithful normal conditional expectation from $M$ onto $M_U$, so that the amenability of $M$ implies the amenability of $M_U$.  In particular, (1) holds since $M_{\{v\}}=M_v$ is amenable for each $v\in V$.  Next, consider two non-adjacent vertices $v$ and $w$.  Then $M_{\{v,w\}} = M_v *M_w$ is amenable, and therefore by \cite[Theorem 2]{ChingFreeProd} one must have $\dim(M_v) = \dim(M_w) = 2$.  Suppose towards a contradiction that there is some vertex $u\in \cV\setminus\{v,w\}$ that is, without loss generality, not adjacent to $w$.  Then $M_{\{u,v,w\}}$ is the free product of $M_u \vee M_v$ and $M_w$ with respect to the appropriate states.  Since $\dim(M_u \vee M_v) \geq 3$ and $\dim(M_w) \geq 2$, $M_{\{u,v,w\}}$ is not amenable by \cite[Theorem 2]{ChingFreeProd}, a contradiction.  Therefore, (2) holds.
\end{proof}

\appendix

\section{Unitaries with state zero}\label{sec: state zero}

For many of our results, it will be convenient to assume that the statial von Neumann algebras attached to the vertices have a unitary in the centralizer algebra with state zero. We note that a related assumption has appeared in \cite[Theorem 2 and Lemma 3]{Bar95} which provides sufficient conditions for a free product of statial von Neumann algebras to be a (possibly type $\mathrm{III}$) factor. In this section, we give a complete characterization of when this occurs and then explore this characterization in a few examples. This characterization is likely folklore, but as we are unable to find a citation in the literature we feel that it is useful to include it for completeness. 
We start with the tracial case, for which we will need the following two lemmas.  

\begin{lem}\label{lem:min acheived}
 For a tracial von Neumann algebra $(M,\tau)$, there exists a $u\in \cU(M)$ with $\tau(u)=\inf|\tau(\cU(A))|$.   
\end{lem}
\begin{proof}
Write $M=M_{1}\oplus M_{2}$ with $M_{1}$ atomic and $M_{2}$ diffuse. Set $K_{i}=\tau(\cU(M_{i}))$ for $i=1,2$. Since $M_{1}$ is atomic and finite, we have that $\cU(M_{1})$ is SOT-compact, so $K_{1}=\tau(\cU(M_{1}))$ is compact. Since $M_{2}$ is diffuse, there is an  embedding of $L^{\infty}([0,1])$ into $M_{2}$ which pulls back $\tau|_{M_{2}}$ to $\tau(1_{M_{2}})$ times integration against Lebesgue measure. This implies that $K_{2}=\tau(\cU(M_{2}))=\{z\in \bC:|z|\leq \tau(1_{M_{2}})\}$, so $K_{2}$ is also compact. Thus
\[\tau(\cU(M))=\{z+w:z\in K_{1},w\in K_{2}\}\]
is the image of the compact space $K_{1}\times K_{2}$ under a continuous map, and so $\tau(\cU(M))$ is compact. The lemma thus follows from continuity of the absolute value map.
\end{proof}

\begin{lem}\label{lem: direct sum trace zero lemma}
Suppose we have tracial von Neumann algebras $(A_{i},\tau_i)_{i=1}^{n}$ and we equip $A = A_1 \oplus\dots \oplus A_n$ with the trace
\[\tau((a_{i})_{i=1}^{n})=\sum_{i=1}^{n}\alpha_{i}\tau_{i}(a_{i}),\]
where $\alpha_1 \geq \alpha_2 \geq \cdots \geq \alpha_n\geq 0$ and $\sum_{i=1}^n \alpha_i = 1$.
Denote $s:=\inf|\tau_{1}(\cU(A_1))|$. Then $|\tau(\cU(A))| = [(\alpha_1(1+s) -1)\vee 0, 1]$.
\end{lem}
\begin{proof}
Let $u\in \cU(A)$ and denote $u=u_1+u_2$ where $u_1\in A_1 $ and $ u_2\in \bigoplus_{i\geq 2} A_i$. Then 
\begin{align*}
    |\tau(u)|\geq &|\tau(u_1)|-|\tau(u_2)| \\
    \geq& \alpha_1s-(\alpha_2+\cdots+ \alpha_k)= \alpha_1s-(1-\alpha_1).
\end{align*}
Hence we have $|\tau(u)|\geq (\alpha_1(s+1)-1) \vee 0$ and $|\tau(\cU(A))| \subset [(\alpha_1(s+1)-1) \vee 0,1]$.

We prove the reverse inclusion by induction on $n$. Since $\cU(A)=\exp(i A_{s.a.})$ is SOT-connected, we have that $|\tau(\mathcal{U}(A))|$ is connected. The case $n=1$ thus follows by connectedness of $|\tau(\mathcal{U}(A))|$ and Lemma \ref{lem:min acheived}.  
We now assume the result true for $n-1$ with $n\geq 2$.  
We split into cases, where in the first case we assume  $\alpha_{1}\geq 1/2$ and thus $\frac{1-\alpha_{1}}{\alpha_{1}}\leq 1$. 
Let $u_{1}\in \cU(A_1)$ with $\tau_{1}(u_{1})= s\vee\left(\frac{1-\alpha_1}{\alpha_1}\right) $, which exists since $|\tau_1(\cU(A_1))|$ is connected and contains $1$. Let $u=(u_1,-1,\ldots, -1)\in \cU(A) $ so that
    \[
        \tau(u)= \alpha_1 \left(s\vee\left(\frac{1-\alpha_1}{\alpha_1}\right) \right)- (1-\alpha_1)= (\alpha_1(1+s)-1)\vee 0.
    \]
Using connectedness again, the claim follows.
In the second case we assume $\alpha_{1}<1/2$, which we note implies $(\alpha_1(1+s)-1)\vee 0 =0$.  Equip $\bigoplus_{i\geq 2}A_{i}$ with the trace $\tau'(a)=\frac{1}{1-\alpha_{1}}\tau(0\oplus a)$ and denote $s':=\inf|\tau_{2}(\mathcal{U}(A_{2}))|$. Note that our inductive hypothesis implies 
    \[
        \left|\tau\left(\cU(\bigoplus_{i\geq 2} A_i)\right)\right| = \left[\left(\frac{\alpha_{2}}{1-\alpha_{1}}(1+s')-1\right)\vee 0, 1 \right].
    \]
Observe that
    \[
        \frac{\alpha_{2}}{1-\alpha_{1}}(1+s')-1\leq \frac{2\alpha_{2}}{1-\alpha_{1}}-1\leq \frac{2\alpha_{1}}{1-\alpha_{1}}-1<  \frac{1}{1-\alpha_{1}}-1=\frac{\alpha_{1}}{1-\alpha_{1}}<1,
    \]
with the last inequality following as $\alpha_{1}<1/2$. Thus we can find $v\in\cU(\bigoplus_{i\geq 2}A_{i})$ with $\tau'(v)=\frac{-\alpha_{1}}{1-\alpha_{1}}$. Then $\tau(1\oplus v)=0$. 

\end{proof}

We now obtain a complete characterization of when a statial von Neumann algebra has a state zero unitary in its centralizer. 

\begin{cor}\label{cor:complete characterization trace zero}
Let $(M,\varphi)$ be a statial von Neumann algebra.
\begin{enumerate}[(i)]
\item Suppose $\varphi$ is a trace and that there exists a non-zero minimal projection $p\in M$ with $\varphi(p)>1/2$. Then $p$ is central and $\varphi(u)\ne 0$ for every $u\in \mathcal{U}(M)$.

\label{label: trace too big gives center}
\item If $\varphi$ is a trace, then there is a $u\in \mathcal{U}(M)$ with $\varphi(u)=0$ if and only if  $\varphi(p)\leq 1/2$ for every minimal projection $p\in M$.\label{item: complete char trace zero case}

\item There exists $u\in \mathcal{U}(M^{\varphi})$ with $\varphi(u)=0$ if and only if $\varphi(p)\leq 1/2$ for every minimal projection $p\in M^{\varphi}$. 
\label{item: complete char state zero case}
\end{enumerate}
\end{cor}
\begin{proof}
(\ref{label: trace too big gives center}):  Suppose that $M$ has a minimal nonzero projection $p$ with $\varphi(p)>1/2$. Let $z$ be the central support of $p$ in $M$. By \cite[Proposition 6.4.3 and Corollary 6.5.3]{KadisonRingroseII}, we have that $Mz$ is isomorphic to $M_{k}(\bC)$ for some $k$. Since $p$ is a minimal projection and $\varphi$ is a trace, it follows that $\varphi(z)=k\varphi(p)$.  Since $\varphi(p)>1/2$, this forces $k=1$. Thus $p=z$ is central.

Since $\bC p$ is a central summand of $M$, using the notation of Lemma~\ref{lem: direct sum trace zero lemma} we have $\alpha_1>1/2$ and $s=1$. Thus this lemma implies $|\varphi(u)|\geq \alpha_1(1+s) - 1 >0$ for all $u\in \cU(M)$.\\ 

\noindent(\ref{item: complete char trace zero case}): The forward implication follows from (\ref{label: trace too big gives center}). For the reverse implication, suppose that $M$ does not have a unitary of trace zero. By decomposing the center into diffuse and atomic parts, we may write 
\[M=M_{0}\oplus \bigoplus_{i\in I}M_{i},\]
where:
\begin{itemize}
    \item $I$ is a countable set (potentially empty),
    \item each $M_{i}$ is a  nonzero finite factor,
    \item $M_{0}$ is either $0$ or has diffuse center.
\end{itemize}
For $j\in I\sqcup\{0\}$ let $\alpha_j:=\varphi(1_{M_j})$ and $\tau_j:= \frac{1}{\alpha_j} \varphi|_{M_j}$ (note that $\alpha_{0}=0$ if $M_{0}=\{0\}$). Since $\sum_{i}\alpha_{i}\leq 1$ we have that either $I$ is finite or $\alpha_{i}\to 0$ as $i\to\infty$ (i.e. as $i$ escapes all finite subsets of $I$). Thus there is an $j_{0}\in I\sqcup \{0\}$ with $\alpha_{j_{0}}=\max\{\alpha_{i}:i\in I\sqcup\{0\}\}$.
Denote $s_{j}:=\inf|\tau_{j}(\mathcal{U}(M_{j})) |$ for each $j\in I\sqcup \{0\}$. By Lemma~\ref{lem: direct sum trace zero lemma}, our hypothesis implies that 
    \[
        \alpha_{j_{0}}(1+s_{j_{0}})>1.
    \]
This inequality implies that $s_{j_{0}}\ne 0$. On other hand, if $M_{0}\ne 0$, then $s_{0}=0$ and for $i\in I$ if $M_{i}$ is a factor of dimension at least $2$, then $s_{i}=0$ as well. So necessarily $M_{j_{0}}\cong \bC 1$ and $s_{j_{0}}=1$. 
But then the above inequality implies $\alpha_{j_{0}}>1/2$ and this proves that $1_{M_{j_0}}\in M$ is a non-zero minimal projection with $\varphi(1_{M_{j_0}})  = \alpha_{j_0} > 1/2$.\\

\noindent(\ref{item: complete char state zero case}): This follows from (\ref{item: complete char trace zero case}), since $\varphi\big|_{M^{\varphi}}$ is a trace. 
\end{proof}

We now list  a few examples of algebras with state zero unitaries in the centralizer.
Let us first consider the matrix algebra case. Suppose that $\varphi$ is a state on $M_{n}(\bC)$. Then we can write 
\[\varphi(x)=\tr(xa),\]
for some $a\in M_{n}(\bC)_{+}$, where $\tr$ is the normalized trace on $M_{n}(\bC)$. Let 
\[a=\sum_{j=1}^{k}\lambda_{j}1_{\{\lambda_{j}\}}(a)\]
be the spectral decomposition of $a$ with $\lambda_{i}\ne \lambda_{j}$ for all $i\ne j$. Set $p_{j}=1_{\lambda_{j}}(a)$. Then  $p_{j}$ is central in $M^{\varphi}$ and $M^{\varphi}p_{j}\cong M_{n\tr(p_{j})}(\bC)$. Suppose $e\in M^{\varphi}$ is a minimal projection. Then we can find a unique $j$ so that $ep_{j}=e$. In this case,
\[\varphi(ep_{j})=\frac{\lambda_{j}}{n \tr(p_{j})}.\]
Hence $M^{\varphi}$ has a state zero unitary if and only if 
    \[
        \lambda\leq \frac{1}{2}\dim(\ker(a-\lambda))
    \]
for every eigenvalue $\lambda$ of $a$.

For general finite dimensional $M$, we may find central projections $z_{1},\cdots,z_{k}$ in $M$ with $\sum_{j=1}^{k}z_{j}=1$ and $Mz_{j}\cong M_{n_{j}}(\bC)$. Let $\tr_{j}$ be the normalized trace on $Mz_{j}$. Then we can find $a_{j}\in (Mz_{j})_{+}$ with $\sum_{j}\varphi(z_{j})\tr_{j}(a_{j})=1$ and
    \[
        \varphi(x)=\sum_{j=1}^{n}\varphi(z_{j})\tr_{j}(xz_{j}a_{j}).
    \]
In this case,
\[M^{\varphi}=\sum_{j=1}^{n}(Mz_{j})^{\tr_{j}(\cdot a_{j})}.\]
If $e\in M^{\varphi}$ is a minimal projection, choose $j$  so that $ez_{j}=e$. Then by the above, there is a $\lambda_{j}$ in the spectrum of $a_{j}$ with
\[\varphi(e)=\varphi(z_{j})\frac{\lambda_{j}}{n\tr_{j}(1_{\lambda_{j}}(a_{j}))}.\]
Hence $M^{\varphi}$ has a state zero unitary if and only if for every $j$ we have
    \[
        \lambda \varphi(z_{j})\leq \frac{1}{2}\dim(\ker(a_{j}-\lambda))
    \]
for every eigenvalue $\lambda$ of $a_{j}$ viewed as an operator on $\bC^{n_{j}}$.

Another example are group von Neumann algebras, equipped with their trace $\tau\colon L(G)\to\bC$ given by $\tau(\lambda_{g})=\delta_{g=e}.$ In this case, any non-trivial group element satisfies the hypotheses. 
Another example would be if $M^{\varphi}$ is diffuse (e.g. $\varphi$ is a trace and $M$ is diffuse). In this case, there is a state-preserving embedding of $L^{\infty}([0,1])$ into $M^{\varphi}$ and so there is a state zero unitary.

\section{Ocneanu Ultrapowers}\label{sec: Ocneaun ultrapowers}

For a cofinal ultrafilter $\omega$ on a directed set $I$ and a von Neumann algebra $M$, denote
    \[
        I_{\omega}(M):=\{(x_{i})_{i\in I}\in \ell^{\infty}(I,M):\lim_{i\to\omega}x_{i}=0 \textnormal{ in the strong-$*$ topology}\},
    \]
    \[
        \cM^{\omega}(M):=\{(x_{i})_{i\in I}\in \ell^{\infty}(I,M):(x_{i})_{i}I_{\omega}(M)+I_{\omega}(M)(x_{i})_{i}\subseteq I_{\omega}(M)\}.
    \]
By \cite{OcneauActions, AndoHaagerup} the quotient $C^{*}$-algebra
    \[
        M^{\omega}:=\cM^{\omega}(M)/I_{\omega}(M)
    \]
is a von Neumann algebra, which we call the \emph{Ocneau ultrapower of $M$}. For $(x_{i})_{i}\in \cM^{\omega}(M)$ we use $(x_{i})_{i\to\omega}$ for its image in $M^{\omega}$. Suppose that $P$ is a subalgebra of $M$ and that there is a faithful normal conditional expectation $E_{P}\colon M\to P$. In this case, $P^{\omega}$ is a naturally a von Neumann subalgebra of $M^{\omega}$ and there is a natural conditional expectation given $E_{P^{\omega}}$ given by
\[E_{P^{\omega}}((x_{i})_{i\to\omega})=(E_{P}(x_{i}))_{i\to\omega}\]
(see \cite[Section 2]{HIBicent} for details).
Applying this with $P=\bC$, we see that if $\varphi$ is a faithful normal state on $M$, then the ultraproduct state $\varphi^{\omega}$ given by
\[\varphi^{\omega}((x_{i})_{i\to\omega})=\lim_{i\to\omega}\varphi(x_{i})\]
remains faithful. This relates to fullness, since \cite[Theorem 5.2]{AndoHaagerup} and \cite[Corollary 3.7]{HMVUltra} show that if $M$ is a $\sigma$-finite von Neumann algebra, then $M'\cap M^{\omega}=\bC$ if and only if $M$ is full.
The following is a statial version of \cite[Lemma 6.1]{CartanAFP}.

\begin{lem}\label{lem:freeproductcenter}
Let $(B,\varphi)$ be a statial von Neumann algebra and let $B\subset M_i$ be an inclusion with expectation $E_i\colon M\to B$ for $i=1,2$. Denote $\varphi_i:=\varphi\circ E_i$ for $i=1,2$ and consider the amalgamated free product $(M,E_B):=(M_1,E_1)*_B (M_2,E_2)$. If there exist unitary elements $u_1\in (M_1)^{\varphi_1}$ and $u_2,u_3\in (M_2)^{\varphi_2}$ such that
\[
E_B[u_1] = E_B[u_2] = E_B[u_3] = E_B[u_2^*u_3] = 0,
\]
then $M'\cap M^\omega \subseteq B^\omega$ for any cofinal ultrafilter $\omega$ on a directed set $I$.
\end{lem}

\begin{proof}
For $i=1,2$ let $L^{2}_{0}(M_{i}):=L^{2}(M_{i},\varphi_i)\ominus L^{2}(B,\varphi)$, and observe that this is the closure of $\{x\in M_{i}:E_{i}(x)=0\}$ in $L^2(M_i,\varphi_i)$. By definition,
    \[
        L^{2}(M,\varphi\circ E_B)=L^{2}(B,\varphi) \oplus \bigoplus_{d=1}^{\infty}\left(\bigoplus_{i_{1}\ne i_{2},\cdots,i_{d-1}\ne i_{d}}L^{2}_{0}(M_{i_{1}})\otimes_{B}\cdots \otimes_{B}L^{2}_{0}(M_{i_{d}})\right).\]
Let $P_{i}$ be the orthogonal projection onto
    \[
        \cH_{i}:=\bigoplus_{d=1}^{\infty}\left(\bigoplus_{i=i_{1}\ne i_{2},\cdots,i_{d-1}\ne i_{d}}L^{2}_{0}(M_{i_{1}})\otimes_{B}\cdots \otimes_{B}L^{2}_{0}(M_{i_{d}})\right).
    \]
Note that $u_1,u_2,u_3\in M^{\varphi\circ E_B}$ since the modular automorphism group of $\varphi\circ E_B$ restricts to that of $\varphi_i$ on $M_i$ for each $i=1,2$. Thus,
    \[
        \|xu_{i}^{*}\|_{2}=\|x\|_{2}
    \]
for all $x\in M$. So right multiplication by $u_{i}^{*}$ extends to a bounded operator on $L^{2}(M,\varphi\circ E_B)$, and we will continue to write $\xi u_{i}^{*}$ for the image of $\xi\in L^{2}(M,\varphi\circ E_B)$ under this operator.
As in \cite[Lemma 6.1]{CartanAFP}, we have
    \[
        u_{1}\cH_{2}u_{1}^{*}\subseteq \cH_{1},\quad  u_{2}\cH_{1}u_{2}^{*}\subseteq \cH_{2},\quad u_{3}\cH_{1}u_{3}^{*}\subseteq \cH_{2},
    \]
and 
    \[ 
        u_{2}\cH_{1}u_{2}^{*}\perp (\cH_{1}+u_{3}\cH_{1}u_{3}^{*}).
    \]
Let $P_{i}$ be the orthogonal projection onto $\cH_{i}, i=1,2$. Note that if $\cK\subseteq L^{2}(M)$ is a closed linear subspace, and $P_{\cK}$ is the orthogonal projection onto $\cK$, then
    \[
        P_{u_{i}\cK u_{i}^{*}}(\cdot)=u_{i}P_{\cK}(u_{i}^{*}\cdot u_{i})u_{i}^{*}.
    \] 
Hence we can argue as in \cite[Lemma 6.1]{CartanAFP} to see that 
    \[
        \|P_{2}(u_{1}\xi u_{1}^{*})\|_{2}\leq \|P_{1}(\xi)\|_{2}\qquad \textnormal{ and }\qquad \|P_{1}(u_{2}\xi u_{2}^{*})\|_{2}^{2}+\|P_{1}(u_{3}\xi u_{3}^{*})\|_{2}^{2}\leq \|P_{2}(\xi)\|_{2}^{2},
    \]
for all $\xi\in L^{2}(M,\varphi\circ E_B)$. Now let $(x_{i})_{i\to\omega}\in M'\cap M^{\omega}$.  Since $\varphi\circ E_{B}$ is faithul, the strong$^{*}$-topology on the unit ball of $M$ coincides with convergence with respect to $\|x\|_{2}+\|x^{*}\|_{2}$ (see \cite[Proposition III.5.3]{TakesakiI}).
We now argue exactly as in \cite[Lemma 6.1]{CartanAFP} to obtain the estimates
    \[
        \lim_{i\to\omega}\|P_{2}(x_{i})\|_{2}\leq \lim_{i\to\omega}\|P_{1}(x_{i})\|_{2}\leq \frac{1}{\sqrt{2}}\lim_{i\to\omega}\|P_{2}(x_{i})\|_{2},
    \]
so that $\lim_{i\to\omega}\|P_{j}(x_{i})\|_{2}=0$ for $j=1,2$. Since $L^{2}(M,\varphi\circ E_B)=L^{2}(B,\varphi)\oplus \cH_{1}\oplus \cH_{2}$, we obtain that $\lim_{i\to\omega}\|x_{i}-E_{B}(x_{i})\|_{2}=0$. By the same argument $\lim_{i\to\omega}\|x_{i}^{*}-E_{B}(x_{i}^{*})\|_{2}=0$, and hence $(x_i)_{i\to\omega} \in B^\omega$. 
\end{proof}

In order for intersections to commute with ultrapowers, it is sufficient to have commuting square inclusions of algebras, as we now show. This is a folklore result, but we give the proof for completeness.

\begin{lem}\label{lem: commuting squares and intersections in ups}
Suppose that $(M,\varphi)$ is a statial von Neumann algebra and that  $M_{1},M_{2}$ are von Neumann subalgebras with $\varphi$-preserving normal conditional expectations $E_{i}\colon M\to M_{i}$. Suppose further that $E_1\circ E_2 = E_2\circ E_1$ so that
\[
\begin{tikzpicture}
    \node at (0,0) {$M_1$};
    \draw[<-] (0.4,0) -- (1.6,0);
    \node at (2,0) {$M$};
    \draw[->] (2,-0.3) -- (2, -0.7);
    \node at (2, -1) {$M_2$};
    \draw[<-] (0.8,-1) -- (1.6, -1);
    \node at (0,-1) {$M_1\cap M_2$};
    \draw[->] (0,-0.3) -- (0, -0.7);
\end{tikzpicture}
\]
forms a commuting square. Then $(M_{1}\cap M_{2})^{\omega}=M_{1}^{\omega}\cap M_{2}^{\omega}$.
\end{lem}

\begin{proof}
Since $E_{M_{i}^{\omega}}=(E_{i})^{\omega},i=1,2$, and similarly $E_{(M_{1}\cap M_{2})^{\omega}}=(E_{M_{1}\cap M_{2}})^{\omega}$, it is enough to show that $E_{M_{1}}|_{M_{2}}=E_{M_{1}\cap M_{2}}$. But this follows from the fact that $E_{M_{1}}\circ E_{M_{2}}=E_{M_{2}}\circ E_{M_{1}}$. 
\end{proof}

\section{Relative amenability}

The following result of Monod--Popa (\cite[Remark 3]{PoMonodRelativeAmen}) allows one to restrict to certain subalgebras when checking relative amenability. We reproduce the well-known proof here.

\begin{lem}\label{lem: commutating squares rel amen reduction}
Suppose that $(M,\varphi)$ is a statial von Neumann algebra and that  $M_{1},M_{2}$ are von Neumann subalgebras with $\varphi$-preserving normal conditional expectations $E_{i}\colon M\to M_{i}$. Suppose further that $E_1\circ E_2 = E_2\circ E_1$ so that
    \[
    \begin{tikzpicture}
    \node at (0,0) {$M_1$};
    \draw[<-] (0.4,0) -- (1.6,0);
    \node at (2,0) {$M$};
    \draw[->] (2,-0.3) -- (2, -0.7);
    \node at (2, -1) {$M_2$};
    \draw[<-] (0.8,-1) -- (1.6, -1);
    \node at (0,-1) {$M_1\cap M_2$};
    \draw[->] (0,-0.3) -- (0, -0.7);
\end{tikzpicture}
    \]
forms a commuting square.
If $M_{1}$ is amenable relative to $M_{1}\cap M_{2}$ (inside $M_{1}$), then $M_{1}$ is amenable relative to $M_{2}$ inside $M$.   
\end{lem}

\begin{proof}
Let $F\colon \ip{M_{1},e_{M_{1}\cap M_{2}}}\to M_{1}$ be a conditional expectation. We will first construct a conditional expectation $E\colon \langle M, e_{M_1\cap M_2} \rangle \to M_1$ that is normal on $M$, so that $M_1$ is amenable relative to $M_1\cap M_2$ inside $M$.

Identify $L^2(M_1,\varphi|_{M_1})$ as a subspace of $L^2(M,\varphi)$ so that the projection onto it is the Jones projection $e_{M_1}$ for the inclusion $(M_1 \subset M,E_1)$. Similarly, the Jones projection $e_{M_1\cap M_2}$ is given by the projection onto the identification of $L^2(M_1\cap M_2,\varphi|_{M_1\cap M_2})$ as a subspace of $L^2(M,\varphi)$. Recall that the basic construction for this inclusion satisfies
    \begin{align}\label{eqn:basic construction formula}
        \langle M, e_{M_1\cap M_2} \rangle = \left( J_\varphi (M_1\cap M_2) J_\varphi \right)' \cap B(L^2(M,\varphi)).
    \end{align}
Consequently, $e_{M_1} \in \langle M, e_{M_1\cap M_2} \rangle$ since $E_{M_1\cap M_2} = E_1\circ E_2$ is $\varphi$-preserving. Additionally, if we define $\Upsilon\colon B(L^2(M,\varphi)) \to B(L^2(M_1,\varphi|_{M_1}))$ by $\Upsilon(T):= e_{M_1} T e_{M_1}$, then $\Upsilon( \langle M, e_{M_1\cap M_2}\rangle) = \langle M_1, e_{M_1\cap M_2} \rangle$\footnote{Here we are abusing notation to let $e_{M_1\cap M_2}$ in the second instance also denote the Jones projection for the inclusion $M_1\cap M_2\subset M_1$. But under the identification $B(L^2(M_1,\varphi|_{M_1})) = e_{M_1} B(L^2(M,\varphi))$ one does have $e_{M_1\cap M_2} e_{M_1} = e_{M_1\cap M_2}$.}. 
Indeed, it is a general fact that if $Q\leq B(\cH)$ is a von Neumann algebra and $p\in Q'\cap B(\cH)$ is a projection, then $(Qp)'\cap B(p\cH)=p(Q'\cap B(\cH))p$. Applying this to $Q= J_\varphi (M_1 \cap M_2) J_\varphi$ and $p= e_{M_1}$ gives
    \begin{align*}
        \Upsilon(\langle M, e_{M_1\cap M_2} \rangle) &= \Upsilon( J_\varphi( M_1\cap M_2) J_\varphi)' \cap B(L^2(M_1,\varphi|_{M_1}))\\
        &= \left( J_{\varphi|_{M_1}} \Upsilon(M_1\cap M_2) J_{\varphi|_{M_1}} \right)' \cap B(L^2(M_1,\varphi|_{M_1}))\\
        &= \left( J_{\varphi|_{M_1}} (M_1\cap M_2) J_{\varphi|_{M_1}} \right)' \cap B(L^2(M_1,\varphi|_{M_1})) = \langle M_1, e_{M_1\cap M_2} \rangle.
    \end{align*}
Now, let $F\colon \ip{M_{1},e_{M_{1}\cap M_{2}}}\to M_{1}$ be a conditional expectation with $F|_{M_1}$ normal, which is guaranteed by $M_1$ being amenable relative to $M_1\cap M_2$ inside $M_1$. Define $E\colon \langle M, e_{M_1\cap M_2} \rangle \to M_1$ by $E:= F\circ \Upsilon$. For $x\in M$ we have $E(x)= F(\Upsilon(x))= \Upsilon(x)$, and if $x\in M_1$ then one further has $E(x)=x$. Thus $E$ is a conditional expectation onto $M_1$ with $E|_M$ normal.

To complete the proof of the lemma, identify $L^2(M_2,\varphi|_{M_2})$ as a subspace of $L^2(M,\varphi)$ and let $e_{M_2}$ be the associated Jones projection for the inclusion $(M_2\subset M, E_2)$. Then (\ref{eqn:basic construction formula}) implies $\langle M, e_{M_{2}}\rangle \leq \langle M, e_{M_1\cap M_2}\rangle$, and so considering the restriction of $E$ to this subalgebra gives that $M_1$ is amenable relative to $M_2$ inside of $M$.
\end{proof}

The next result provides a sufficient condition for preventing amalgamated free products from being amenable relative to either of the factors; see e.g. \cite{KuroshType, DKEP21} for similar arguments.

\begin{lem} \label{lem: free product non-amenability}
Let $(B,\varphi)$ be a statial von Neumann algebra and let $B\subset M_i$ be an inclusion with expectation $E_i\colon M\to B$ for $i=1,2$. Denote $\varphi_i:=\varphi\circ E_i$ for $i=1,2$ and consider the amalgamated free product $(M,E_B):=(M_1,E_1)*_B (M_2,E_2)$. If there exist unitary elements $u_0\in (M_1)^{\varphi_1}$, $u_1\in M_1$, and $u_2\in (M_2)^{\varphi_2}$ such that
    \[
        E_B[u_0] = E_B[u_1] = E_B[u_0^*u_1] = E_B[u_2] = 0,
    \]
then $M$ is not amenable to relative to either $M_1$ or $M_2$.
\end{lem}

\begin{proof}
Define $\psi:=\varphi\circ E_B$. For each $j=1,2$, denote $H_j:=L^2(M_j,\varphi_j)$, which we identify as a subspace of $L^2(M,\psi)$. Also identify $L^2(B,\varphi)$ as a subspace of $L^2(M,\psi)$ and denote $H_j^\circ := H_j\ominus L^2(B,\varphi)$, $j=1,2$. Recall that
    \[
        L^2(M,\psi) = \bigoplus_{d \in \bN} \bigoplus_{i_1 \neq \dots \neq i_d} H_{i_1}^\circ \otimes_B \dots \otimes_B H_{i_d}^\circ.
    \]
Denote
    \[
        K := \bigoplus_{d \in \bN} (H_1^\circ \otimes_B H_2^\circ)^{\otimes_B d} \otimes_B H_1,
    \]
so that as right $M_1$-modules we have
    \[
        L^2(M,\psi )_{M_1} \cong K_{M_1} \oplus \left( H_2^\circ \otimes_B K\right)_{M_1}.
    \]
In particular, we can identify $K^\perp$ with $H_2^\circ \otimes_B K$. By assumption, $u_0, u_1\in H_1^\circ$ so that $u_0 K^\perp, u_1 K^\perp \leq K$, and since $u_0^* u_1\in H_1^\circ$ we further have that $u_0 K^\perp \perp u_1 K^\perp$. Thus if $P_K, P_{K^\perp}\in B(L^2(M,\psi))$ denote the projections onto $K$ and $K^\perp$, respectively, then
    \[
        u_0 P_{K^\perp} u_0^* + u_1 P_{K^\perp} u_1^* \leq P_K.
    \]
Additionally, $u_2\in H_2^\circ$ implies $u_2 K \leq K^\perp$ so that
    \begin{align}\label{eqn:paradox decomp inequality}
        u_2 P_K u_2^* \leq P_{K^\perp}.
    \end{align}
Now, suppose, towards a contradiction, that there exists a conditional expectation $\Phi\colon \langle M,e_{M_1} \rangle \to M$. Note that $P_K,P_{K^\perp} \in \langle M, e_{M_1}\rangle$ since $K$ and $K^\perp$ are invariant for $J_\psi M_1 J_\psi$, and so the above inequalities imply
    \[
        \psi( u_0 \Phi(P_{K^\perp}) u_0^*)  + \psi(u_1 \Phi(P_{K^\perp}) u_1^*) \leq \psi(\Phi(P_K))
    \]
and
    \[
        \psi(u_2 \Phi(P_K) u_2^*) \leq \psi(\Phi(P_{K^\perp})).
    \]
Recall that $u_0 \in (M_1)^{\varphi_1}$ and $u_2\in (M_2)^{\varphi_2}$ so that $u_0, u_2\in M^{\psi}$ and hence
    \begin{align*}
        \psi(\Phi(P_{K^\perp})) + \psi(u_1 \Phi(P_{K^\perp}) u_1^*) = \psi(u_0\Phi(P_{K^\perp})u_0^*) + \psi(u_1 \Phi(P_{K^\perp}) u_1^*)&\leq \psi(\Phi(P_K))\\
        &= \psi(u_2 \Phi(P_k) u_2^*)\leq \psi(\Phi(P_{K^\perp})).
    \end{align*}
Hence $\psi(u_1\Phi(P_{K^\perp})u_1^*)=0$, and therefore $\Phi(P_{K^\perp})=0$. Since, (\ref{eqn:paradox decomp inequality}) implies $u_2 \Phi(P_K) u_2^* \leq \Phi(P_{K^\perp})=0$, we also have $\Phi(P_K)=0$. But this leads to the contradiction
    \[
        \Phi(1)= \Phi(P_K) + \Phi(P_{K^\perp})=0.
    \]
Thus $M$ is not amenable relative to $M_1$.

To see that $M$ is not amenable relative to $M_2$, denote
    \[
        L:= \bigoplus_{d\in \bN} (H_2^\circ \otimes_B H_1^\circ)^{\otimes_B d}\otimes_B H_2
    \]
so that as right $M_2$-modules we have
    \[
        L^2(M,\psi)_{M_2} \cong L_{M_2} \oplus (H_1^\circ \otimes_B L)_{M_2}.
    \]
Then $u_0 L$ and $u_1 L$ are orthogonal subspaces in $L^\perp = H_1^\circ \otimes_B L$, and $u_2 L^\perp \leq L$. So if one assumes there exists a conditional expectation from $\langle M,e_{M_2}\rangle $ to $M$, then we can proceed as above to obtain a contradiction.
\end{proof}

In contrast to the previous lemma, the next result shows that tensoring relatively amenable inclusions yields a relatively amenable inclusion.

\begin{lem}\label{lem: tensor rel amen}
For $i=1,2$, let $N_i\leq M_i$ be an inclusion of von Neumann algebras admitting faithful normal conditional expectations $E_i\colon M_i\to N_i$. If $M_i$ is amenable relative to $N_i$ for each $i=1,2$, then $M_1\bar\otimes M_2$ is amenable relative to $N_1\bar\otimes N_2$.
\end{lem}

\begin{proof}
By Tomita's commutation theorem \cite[Theorem IV.5.9]{TakesakiI}, we have a canonical isomorphism
    \[
        \ip{M_{1}\bar{\otimes}M_{2},e_{N_{1}\bar{\otimes}N_{2}}}\cong \ip{M_{1},e_{N_{1}}}\bar{\otimes}\ip{M_{2},e_{N_{2}}}
    \]
satisfying
    \[
        (x_{1}\otimes x_{2})e_{N_{1}\bar{\otimes}N_{2}}(x_{2}\otimes y_{2})\mapsto (x_{1}e_{N_{1}}y_{1})\otimes(x_{2}e_{N_{2}}y_{2}).
    \]
By assumption, there are conditional expectations $\Phi_{i}\colon \ip{M_{i},e_{N_{i}}}\to M_{i}$ for $i=1,2$. Since these expectations are not normal, $\Phi_1\otimes 1$ and $1\otimes \Phi_2$ need not extend from the algebraic tensor product $\ip{M_1, e_{N_1}}\odot \ip{M_2,e_{N_2}}$ to $\ip{M_1\bar\otimes M_2, e_{N_1\bar\otimes N_2}}$, and hence $(\Phi_1\otimes 1)\circ (1\otimes \Phi_2)$ need not be a conditional expectation onto $M_1\bar\otimes M_2$. However, one can argue abstractly using the following claim.\\

\noindent\textbf{Claim:} Let $B\leq S$ be von Neumann algebras and let $\cH$ be any Hilbert space. If $\Phi\colon S\to B$ is a conditional expectation, then there is a conditional expectation $\widetilde{\Phi}\colon S\bar{\otimes}B(\cH)\to B\bar{\otimes}B(\cH)$. Moreover, if $T\leq B(\cH)$ is a von Neumann algebra, then $\widetilde{\Phi}|_{S\overline{\otimes}T}$ is a conditional expectation onto $B\overline{\otimes}T$.

To prove the claim, view $S\subseteq B(\cK)$ and  let $(e_{i})_{i\in I}$ be an orthonormal basis for for $\cH$. Let $\omega_{i,j}\in B(\cH)$ denote the rank one operator $\omega_{i,j}(\xi):=\ip{\xi,e_{j}}e_{i}$. For $A\in S\bar{\otimes}B(\cH)$ with $A=\sum_{i,j}A_{ij}\otimes \omega_{i,j}$ (with the sum converging in the strong operator topoology). We wish to define $\widetilde{\Phi}(A)$ as
    \[
        \sum_{i,j\in I} \Phi(A_{ij})\otimes \omega_{i,j},
    \]
in which case $\widetilde{\Phi}$ being a conditional expectation follows directly from $\Phi$ being a conditional expectation. To see that the above sum converges in the strong operator topology, for a  finite subset $F\subset I$ denote $p_F:=\sum_{i\in F} 1\otimes \omega_{i,i}$ and observe that
    \[
        \left\|\sum_{i,j\in F}\Phi(A_{ij})\otimes \omega_{i,j} \right\|\leq\left\|\sum_{i,j\in F} A_{ij}\otimes \omega_{i,j}\right\| = \left\|p_F\left( \sum_{i,j\in I} A_{ij}\otimes \omega_{i,j}\right) p_F\right\| \leq \left\|\sum_{i,j\in I} A_{ij}\otimes \omega_{i,j}\right\|,
    \]
where the first inequality follows from $\Phi$ being completely bounded with $\|\Phi\|_{cb}=1$. Thus the sum defining $\tilde{\Phi}(A)$ converges in the strong operator topoology since the net $\left(\sum_{i,j\in F}\Phi(A_{ij})\otimes \omega_{i,j}\right)_{F\Subset I}$ converges pointwise on the dense subspace $\text{span}\{ \xi\otimes e_i\colon \xi\in \cK,\ i\in I\}$ and is uniformly bounded in norm.

For the second part of the claim, let $y\in T'$. Then since $1\bar{\otimes}B(\cH)$ is in the multiplicative domain of $\widetilde{\Phi}$, we have that $1\otimes y\in \widetilde{\Phi}(S\bar{\otimes}T )'$. Tomita's commutation theorem thus shows that: 
    \[
        \widetilde{\Phi}(S\bar{\otimes}T )\subseteq (1\otimes T')' \cap (B\bar{\otimes}B(\cH))=B\bar{\otimes}T.
    \]
This proves the claim.\\

Applying the claim first to $\Phi_1\colon \ip{M_1,e_{N_1}}\to M_1$ and $T=N_2$ yields a conditional expectation 
    \[
        \tilde{\Phi}_1\colon \ip{M_1,e_{N_1}}\bar\otimes N_2 \to M_1\bar\otimes N_2.
    \]
Then applying the claim next to $\Phi_2\colon \ip{M_2,e_{N_2}}\to M_2$ and $T=\ip{M_1, e_{N_1}}$ yields a conditional expectation
    \[
        \tilde{\Phi}_2 \colon \ip{M_1,e_{N_1}} \bar\otimes \ip{M_2, e_{N_2}} \to \ip{M_1,e_{N_1}} \bar\otimes M_2.
    \]
Hence $\tilde{\Phi}_2 \circ \tilde{\Phi}_1$ is a conditional expectation that witnesses the relative amenability of $M_1\bar\otimes M_2$ to $N_1\bar\otimes N_2$.
\end{proof}

\bibliographystyle{amsalpha}
\bibliography{graphproducts}

\end{document}